\documentclass[11pt]{amsart}
\usepackage{amssymb}
\usepackage{latexsym} 
\usepackage{srcltx} 
\usepackage{esint}
\usepackage{xcolor}

\def\nc{\newcommand}

 \def\Om{\Omega}

\nc\pa{\partial}

\nc\CC{\mathbb{C}}
\nc\RR{\mathbb{R}}
\nc\QQ{\mathbb{Q}}
\nc\ZZ{\mathbb{Z}}
\nc\NN{\mathbb{N}}

\nc\m[1]{\left| #1\right|}
\nc\norm[1]{\left\| #1\right\|}
\newtheorem{theorem}{Theorem}[section]
\newtheorem{lemma}[theorem]{Lemma}

\newtheorem{proposition}[theorem]{Proposition}
\newtheorem{definition}[theorem]{Definition}
\newtheorem{remark}[theorem]{Remark}        
\numberwithin{equation}{section}

\begin{document}

\title[]{Characterizations of predual spaces to a class of Sobolev multiplier type spaces}

\author[Keng Hao Ooi]
{Keng Hao Ooi}
\address{Department of Mathematics,
Louisiana State University,
303 Lockett Hall, Baton Rouge, LA 70803, USA.}
\email{kooi1@lsu.edu}

\author[Nguyen Cong Phuc]
{Nguyen Cong Phuc}
\address{Department of Mathematics,
Louisiana State University,
303 Lockett Hall, Baton Rouge, LA 70803, USA.}
\email{pcnguyen@math.lsu.edu}


\begin{abstract} We characterize preduals and  K\"othe duals to a class of  Sobolev multiplier type spaces.  Our results fit in well with the modern  theory of function spaces of harmonic analysis and are also applicable to nonlinear partial differential equations.  We make use of several tools from nonlinear potential theory, weighted norm inequalities, and 	the theory of Banach function spaces to obtain our results.   
\end{abstract}

\maketitle

\section{Introduction}

Let $\alpha$ be a real number and $s>1$. We define the space of Bessel potentials $H^{\alpha, s}=H^{\alpha, s}(\mathbb{R}^n)$, $n\geq 1$, 
as the completion of $C_c^\infty(\mathbb{R}^n)$ with respect to the norm 
$$\|u\|_{H^{\alpha,s}}=\|(1-\Delta)^{\frac{\alpha}{2}}u \|_{L^s(\mathbb{R}^n)}.$$ 
Here the operator $(1-\Delta)^{\frac{\alpha}{2}}$ is understood as
$(1-\Delta)^{\frac{\alpha}{2}}:=\mathcal{F}^{-1}(1+|\xi|^2)^{\frac{\alpha}{2}}\mathcal{F},$ 
where $\mathcal{F}$ is the Fourier transform in $\mathbb{R}^n$. In the case $\alpha>0$, it follows  that (see, e.g., \cite{MH}) a function $u$ belongs to $H^{\alpha,s}$ if and only if
$$u= G_{\alpha}* f$$ 
for some $f\in L^s$, and moreover
$\|u\|_{H^{\alpha, s}}=\|f\|_{L^s}.$ 
Here $G_\alpha$ is  the Bessel kernel of order $\alpha$ defined by $G_{\alpha}(x):= \mathcal{F}^{-1}[(1+|\xi|^2)^{\frac{-\alpha}{2}}](x)$.

The Bessel potential  space $H^{\alpha,s}$, $\alpha>0, s>1$, can be viewed as a fractional generalization  of the standard Sobolev space $W^{k,s}=W^{k,s}(\mathbb{R}^n)$, $k\in \mathbb{N}, s>1$.  The latter, by definition, consists of functions in $L^s$ whose distributional derivatives up to order $k$ also belong to $L^s$. The norm of a function $u\in W^{k,s}$ is given by  $\|u\|_{W^{k,s}}=\sum_{|\beta|=k}\|D^{\beta}u\|_{L^{s}} + \|u\|_{L^{s}}.$
Indeed, it follows from the theory of singular integrals that for any $k\in \mathbb{N}$ and $s>1$ we have $H^{k,s}\approx W^{k,s}$, i.e., there exists a constant $A>0$ such that
\begin{align}\label{WH}
A^{-1}\|u\|_{H^{k,s}}  \leq \|u\|_{W^{k,s}}\leq A \|u\|_{H^{k,s}}.
\end{align}

In this paper we are concerned with the Banach space $M^{\alpha,s}_p=M^{\alpha,s}_p(\mathbb{R}^n)$, $\alpha>0, s>1, p\geq 1$, 
which consists of functions $f\in L^p_{\rm loc}(\mathbb{R}^n)$
such that the trace inequality 
\begin{align}\label{trace}
\left(\int_{\mathbb{R}^n} |u|^s |f|^p dx\right)^{\frac{1}{p}} \leq C \|u\|_{H^{\alpha,s}}^{\frac{s}{p}}
\end{align} 
holds for all $u\in C_c^\infty(\mathbb{R}^n)$. A norm  of a function $f\in M^{\alpha,s}_p$  is defined as the least possible constant $C$ in the above inequality. Of course,  inequality \eqref{trace}  can be equivalently written as 
\begin{align*}
\left(\int_{\mathbb{R}^n} (G_\alpha*h) ^s |f|^p dx\right)^{\frac{1}{p}} \leq C \|h\|_{L^s(\RR^n)}^{\frac{s}{p}},
\end{align*} 
for all nonnegative $h\in L^s(\RR^n)$.

We note that  the space $M^{\alpha,s}_p$ or its homogeneous counterpart appears naturally in many super-critical nonlinear PDEs including the Navier-Stokes system (see, e.g., \cite{VW, KV, HMV, Ph2, Ph3, PhV, NP, AP2, PhPh, L-R, Ger}). It also follows from the  definition that  a function   $f\in M^{\alpha,s}_p$ if and only if $|f|^{p/s}$ belongs to  the space of Sobolev multipliers $M(H^{\alpha,s}\rightarrow L^s)$ that has been studied in  \cite{MS1, MS2}.

The main goal of this paper is to find  `good' predual spaces  to  the space $M^{\alpha,s}_p$, $p>1$.  By a good predual space in this context we mean one that fits in well with the theory of function spaces of harmonic analysis and partial differential equations. For example, 
one should be able to demonstrate the behavior of basic operators such as the Hardy-Littlewood maximal function and Calder\'on-Zygmund  operators on such a space.  
A natural candidate for a predual of $M^{\alpha,s}_p$ is its K\"othe dual space $(M^{\alpha,s}_p)'$ defined by
\begin{equation}\label{KothDef}
(M^{\alpha,s}_p)'=\left\{{\rm measurable~ functions~} f: \sup\int|fg|dx<+\infty\right\},
\end{equation}
where the supremum is taken over all functions $g$ in the unit ball of $M^{\alpha,s}_p$. The norm of $f\in (M^{\alpha,s}_p)'$ is defined as the above 
supremum. Indeed, as in \cite{KV}, using the $p$-convexity of $M^{\alpha,s}_p$   we find that this is the case , i.e.,
$$[(M^{\alpha,s}_p)']^*=M^{\alpha,s}_p,$$
(see Proposition \ref{M'*} below).

We observe however that the space $(M^{\alpha,s}_p)'$ is quite abstract and thus it is desirable to find a more concrete space that is isomorphic to it. In this paper, inspired from the  work  \cite{AX1, AX2}, several other predual spaces to $M^{\alpha,s}_p$ will be constructed. 
In particular, we find a Banach  function space (in the sense of \cite{Lux}) $\mathcal{N}_{p'}^{\alpha, s}$, $p'=p/(p-1)$, such that     
$(M^{\alpha,s}_p)' \approx \mathcal{N}_{p'}^{\alpha, s}$ for all $p>1, \alpha>0$, and $\alpha p\leq n$. 
More importantly, the space  $\mathcal{N}_{p'}^{\alpha, s}$ and  other predual spaces that we construct have a nice structure that  the  Hardy-Littlewood maximal function and  standard Calder\'on-Zygmund type operators behave well on them in a reasonable sense. As a result, the {\it local} Hardy-Littlewood maximal function ${\bf M}^{\rm loc}$ (see \eqref{localmaxdef} below) is shown to be bounded on 
$(M^{\alpha,s}_p)'$. We  remark that whereas the Hardy-Littlewood maximal function ${\bf M}$ is bounded on $M^{\alpha,s}_p$ for  any $p>1$ and $\alpha s\leq n$ (see \cite{MS1}), 
it fails to be bounded on $(M^{\alpha,s}_p)'$. This is because $L^\infty(\RR^n)\hookrightarrow M^{\alpha,s}_p$ (see \eqref{LinftyM} below)  and thus $ (M^{\alpha,s}_p)' \hookrightarrow L^1(\RR^n)$. This phenomenon happens simply because of  the inhomogeneity of the Sobolev space under consideration. In the homogeneous case, where the space $H^{\alpha,s}$ in \eqref{trace} is replaced with its homogeneous counterpart $\dot{H}^{\alpha,s}$ (the space of Riesz potentials),  such a phenomenon does not  exist. We shall discuss the homogeneous case later in 
Section \ref{homogeneoussetting} of the paper.

Our approach to the preduals of $M^{\alpha,s}_p$ is based upon the notion of Bessel  capacity associated to the Bessel potential space 
$H^{\alpha, s}$.  For each set $E\subset\mathbb{R}^n$, the Bessel  capacity of $E$ is defined by 
\begin{equation}\label{BC}
{\rm Cap}_{\alpha, \,s}(E):=\inf\Big\{\|f\|_{L^{s}}^{s}: f\geq0, G_{\alpha}*f\geq 1 {\rm ~on~} E \Big\}.
\end{equation}

Note that functions in  $H^{\alpha, s}$ are generally not continuous. One can think of 
${\rm Cap}(\cdot)$ as a device to measure the discontinuity of functions in $H^{\alpha, s}$, especially when $\alpha s\leq n$. For example, one has the following Lu sin type theorem: If $f\in H^{\alpha, s}$, then $f$ has a  quasicontinuous representative $\tilde{f}$. That is,  $f=\tilde{f}$ a.e. and  $\tilde{f}$ is quasicontinuous
with respect to  the capacity ${\rm Cap}_{\alpha, s}$. By definition, a function $\tilde{f}$ is said to be quasicontinuous
with respect to  ${\rm Cap}_{\alpha, s}$ if for any $\epsilon>0$ there exists an open set $G$ such that ${\rm Cap}_{\alpha,s}(G)<\epsilon$ and $\tilde{f}$ is continuous in $G^c:=\mathbb{R}^n\setminus G$ (see \cite{AH}).

Originally appearing in electrostatics, capacities have found many applications in analysis and PDEs. For example, they are used to study the pointwise behavior of Sobolev functions, removable  singularities of solutions to linear and non-linear PDEs, and Dirichlet problems on arbitrary domains (Wiener's criterion), etc. In this paper, the importance of capacity lies in the following theorem (see, e.g., \cite{MS2, AH}):

\begin{theorem}[Maz'ya-Adams-Dahlberg] \label{MAD} Let    $\alpha>0$, $s>1$ and suppose that $\nu$ is a nonnegative locally finite measure in $\mathbb{R}^n$. Then the following properties are equivalent: 
	
	{\rm (i)} The inequality 
	\begin{equation*}
	\int_{\mathbb{R}^n} (G_{\alpha}*f)^{s}d\nu \leq A_1 \int_{\mathbb{R}^n} f^{s}dx
	\end{equation*}
	holds for all functions $f\in L^{s}(\mathbb{R}^n), f\geq 0$.
	
	{\rm (ii)} The inequality
	\begin{equation*}
	\nu(K)\leq A_2\, {\rm Cap}_{\alpha,  s}(K) 
	\end{equation*}
	holds for all compact sets  $K\subset \mathbb{R}^n$.
	
	{\rm (iii)} The weak-type inequality 
	\begin{equation*}
	\sup_{t>0} t^s \nu(\{x\in \mathbb{R}^n: G_{\alpha}*f (x)>t\})  \leq A_3 \int_{\mathbb{R}^n} f^{s}dx 
	\end{equation*}
	holds for all functions $f\in L^{s}(\mathbb{R}^n), f\geq 0$.
	
	Moreover, the least possible values of $A_i$, $i=1,2,3$, are equivalent.
\end{theorem}

As an application of Theorem \ref{MAD}, we find that the norm of a function $f\in M^{\alpha,s}_p$ is equivalent to the quantity
\begin{align}\label{sobolev}
\sup_{K}\left(\frac{\int_{K}|f(x)|^{p}dx}{\text{Cap}_{\alpha,s}(K)}\right)^{1/p},
\end{align}
where the supremum is taken over all compact sets $K\subset{\mathbb{R}}^{n}$ with non-zero capacity. 
Thus in what follows,  we shall tactically use \eqref{sobolev} as the norm for functions $f$ in $M^{\alpha,s}_p$, i.e., we redefine
$$\|f\|_{M^{\alpha,s}_p}:=\sup_{K}\left(\frac{\int_{K}|f(x)|^{p}dx}{\text{Cap}_{\alpha,s}(K)}\right)^{1/p}, \qquad p\geq 1.$$
 
It is worth mentioning that one also has (see \cite[Remark 3.1.1]{MS2}):
$$\|f\|_{M^{\alpha,s}_p} \simeq\sup_{K:\, {\rm dian}(K)\leq 1}\left(\frac{\int_{K}|f(x)|^{p}dx}{\text{Cap}_{\alpha,s}(K)}\right)^{1/p}.$$
Thus in view of \eqref{alphasn} below,  we have 
$\|f\|_{M^{\alpha,s}_p} \simeq \sup_{x\in \RR^n} \|f\|_{L^{p}(B_1(x))}$ provided  $\alpha s> n$. That is, when $\alpha s> n$, $M^{\alpha,s}_p$ can be identified with the space of uniformly local $L^p$ functions in $\RR^n$.
 For this reason, 
 we shall be mainly interested in   the case $\alpha s\leq n$. On the other hand, for $\alpha s<n$ by \eqref{rleq1} below we see that $M^{\alpha,s}_p$ is continuously embedded into a local Morrey space.

Motivated from (ii) of Theorem \ref{MAD},  we  also define $\mathfrak{M}^{\alpha,s}=\mathfrak{M}^{\alpha,s}(\mathbb{R}^n)$ as the space  of  locally finite signed measure $\mu$ in $\mathbb{R}^n$ such that the norm $\|\mu\|_{\mathfrak{M}^{\alpha,s}}<+\infty$, where 
\begin{align*}
\|\mu\|_{\mathfrak{M}^{\alpha,s}}:=\sup_{K}\dfrac{|\mu|(K)}{\text{Cap}_{\alpha,s}(K)}
\end{align*}
with the supremum being taken over all compact sets $K\subset{\mathbb{R}}^{n}$ such that $\text{Cap}_{\alpha,s}(K)\not=0$.
It is obvious that  $M^{\alpha,s}_1$ is continuously embedded into $\mathfrak{M}^{\alpha,s}$.

The notion of Choquet integral will be important in this work. Let $w:\RR^n \rightarrow [0,\infty]$ be defined  $\text{Cap}_{\alpha,s}$-quasieverywhere, i.e., defined except for only a set of 
zero  capacity $\text{Cap}_{\alpha,s}$. The Choquet integral of $w$ is defined by 
\begin{equation}\label{ChoI}
\int_{\RR^n} w d C:=\int_{0}^{\infty}\text{Cap}_{\alpha,s}(\{x\in\mathbb{R}^n: w(x)>t\})dt.
\end{equation}

We let  $L^1(C)$ be the space of quasicontinuous (hence  quasieverywhere defined) functions $f$ in $\mathbb{R}^n$ such that  
\begin{align}\label{L1C}
\|f\|_{L^1(C)} := \int |f| d C <+\infty.
\end{align}

In general, $L^1(C)$ is only a quasi-Banach space (see Proposition \ref{completeness} below) as $\|\cdot\|_{L^1(C)}$ may not satisfy the triangle inequality. However, by a theorem of Choquet (see \cite{Cho, Den}), $\|\cdot\|_{L^1(C)}$ satisfies the triangle inequality (hence $L^1(C)$ is a Banach space) if and only if the associated capacity $\text{Cap}_{\alpha,s}$ is strongly subadditive. By definition, the  capacity $\text{Cap}_{\alpha,s}$ is strongly subadditive  if for  any two sets $E_{1},E_{2}\subset{\mathbb{R}}^{n}$, 
\begin{align}\label{StrongS}
\text{Cap}_{\alpha,s}(E_{1}\cup E_{2})+\text{Cap}_{\alpha,s}(E_{1}\cap E_{2})\leq\text{Cap}_{\alpha,s}(E_{1})+\text{Cap}_{\alpha,s}(E_{2}).
\end{align} 

It is known that $\text{Cap}_{\alpha,2}$, $0<\alpha\leq 1$ is strongly subadditive, and $\text{Cap}_{1,s}$, $s>1$, is equivalent to one that is  strongly subadditive
(see Section \ref{Pre}).

Our first result provides another equivalent norm for the space   $M^{\alpha,s}_p$, $p>1$.

\begin{theorem}\label{firsttheorem}
	For $p>1$ and $\alpha>0, s>1$, with $\alpha s\leq n$, we have 
	\begin{align*}
	\|f\|_{M^{\alpha,s}_p}\simeq \sup_{w}\left(\int_{{\RR}^{n}}|f(x)|^{p} w(x)dx\right)^{1/p},
	\end{align*}
	where the supremum is taken over all nonnegative  $w\in L^{1}(C)\cap A^{\rm loc}_1$ with $\|w\|_{L^{1}(C)}\leq 1$ and $[w]_{A^{\rm loc}_1}\leq {\bf \overline{c}}(n,\alpha,s)$ for a constant ${\bf \overline{c}}(n,\alpha,s)\geq 1$ that depends only on  $n,\alpha$ and $s$.
\end{theorem}

Here   $A^{\rm loc}_1$ is the class of local $A_1$ weights which consists of nonnegative locally integrable functions $w$ in $\mathbb{R}^n$ such that
\begin{equation}\label{Mloc}
{\bf M}^{\rm loc} w(x)\leq C w(x)
\end{equation}
for a.e. $x\in\mathbb{R}^n$. The $A^{\rm loc}_1$ characteristic constant of $w$, $[w]_{A^{\rm loc}_1}$, is defined as the least possible constant $C$ in the above inequality.
The operator ${\bf M}^{\rm loc}$ stands for the (center) local Hardy-Littlewood maximal function defined for each $f\in L^1_{\rm loc}(\mathbb{R}^n)$ by 
 \begin{equation}\label{localmaxdef}
 {\bf M}^{\rm loc} f (x)= \sup_{0<r\leq 1}   \frac{1}{|B_r(x)|} \int_{B_{r}(x)} |f(y)|dy.
\end{equation}
We recall that the (center) Hardy-Littlewood maximal function ${\bf M}f$ of $f$ is defined similarly except that the supremum  is now taken over all $r>0$. If  \eqref{Mloc} holds a.e. with ${\bf M}$ in place of  ${\bf M}^{\rm loc}$, then we say that $w$ belongs to the class $A_1$.   

One should relate Theorem \ref{firsttheorem} to  \cite[Theorem 2.2]{AX1} and \cite[Lemma 11]{AX2} in the context of (homogeneous) Morrey spaces.  Here we mention that our approach to Theorem \ref{firsttheorem} actually provides a new proof of  \cite[Lemma 11]{AX2} in which the result of \cite{OV} can be completely avoided.

Inspired by Theorem \ref{firsttheorem} we define the following space. For $q>1$ and $\alpha>0, s>1$,  let $N^{\alpha,s}_q=N^{\alpha,s}_q(\mathbb{R}^n)$
be the space of all measurable functions $g$ such that there exists a weight $w\in L^{1}(C)\cap A^{\rm loc}_1$ with $\|w\|_{L^{1}(C)}\leq 1$ and $[w]_{A^{\rm loc}_1}\leq {\bf \overline{c}}(n,\alpha,s)$ such 
that 
$$\left(\int_{\mathbb{R}^n}|g(x)|^{q}w(x)^{1-q}dx\right)^{1/q}<+\infty.$$

This implies that, for such $w$, $g=0$ a.e. on the set $\{w=0\}$.
The `norm' of a function $g\in N^{\alpha,s}_q$ is the defined as
\begin{align*}
\|g\|_{N^{\alpha,s}_q}=\inf_{w}\left(\int_{\mathbb{R}^n}|g(x)|^{q}w(x)^{1-q}dx\right)^{1/q},
\end{align*}
where the infimum is taken over all $w\in L^{1}(C)\cap A^{\rm loc}_1$ with $\|w\|_{L^{1}(C)}\leq 1$ and $[w]_{A^{\rm loc}_1}\leq {\bf \overline{c}}(n,\alpha,s)$.

Our first duality result can now be stated. 
\begin{theorem}\label{second} Let  $p>1$,  $\alpha>0, s>1$, with $\alpha s\leq n$,  and  $p'=p/(p-1)$. We have 
	$$\left(N^{\alpha,s}_{p'}\right)^{\ast} \approx M^{\alpha,s}_{p}$$ in the sense that each bounded linear functional $L\in \left(N^{\alpha,s}_{p'}\right)^{\ast}$ corresponds to a unique 
	$f\in M^{\alpha,s}_p$ such that $L(g)=L_f(g)$ for all $g\in N^{\alpha,s}_{p'}$, where 
	\begin{align*}
	L_{f}(g)=\int_{\mathbb{R}^n}f(x)g(x)dx,\qquad g\in N^{\alpha,s}_{p'}.
	\end{align*}

Moreover, 	we have 
	\begin{align*}
	\|f\|_{M^{\alpha,s}_{p}}\simeq\|L_{f}\|_{\left(N^{\alpha,s}_{p'}\right)^{\ast}}.
	\end{align*} 
\end{theorem}

In Theorems \ref{firsttheorem} and \ref{second}, we can also drop the $A^{\rm loc}_1$  and the quasicontinuity conditions on the weights $w$ and obtain the following similar results with equality of norms.
\begin{theorem}\label{NoA1Th}
For $p>1$ and $\alpha>0, s>1$, with $\alpha s\leq n$, we have 
\begin{align*}
\|f\|_{M^{\alpha,s}_p} = \sup_{w}\left(\int_{{\RR}^{n}}|f(x)|^{p} w(x)dx\right)^{1/p},
\end{align*}
where the supremum is taken over all weights  $w$ such that $w$ is defined ${\rm Cap}_{\alpha,s}$-quasieverywhere and 
$\int_{\RR^n} w dC \leq 1$.

Moreover, we have  $\left(\widetilde{N}^{\alpha,s}_{p'}\right)^{\ast}=M^{\alpha,s}_p$, where  $\widetilde{N}^{\alpha,s}_q=\widetilde{N}^{\alpha,s}_q(\mathbb{R}^n)$, $q>1$,
is the space of all measurable functions $g$ such that 
\begin{align*}
\|g\|_{\widetilde{N}^{\alpha,s}_q}:=\inf_{w}\left(\int_{\mathbb{R}^n}|g(x)|^{q}w(x)^{1-q}dx\right)^{1/q}<+\infty,
\end{align*}
Here the infimum is taken over all nonnegative q.e. defined function $w\in L^1_{\rm loc}(\RR^n)$ such that  $\int_{\RR^n} w dC \leq 1$. 
\end{theorem}

The spaces $N^{\alpha,s}_q$ and $\widetilde{N}^{\alpha,s}_q$ are obviously quasinormed spaces. However, at this point it is not clear if they are normable or complete for all $\alpha>0, s>1$ with $\alpha s\leq n$ and $q>1$.
 We now introduce two  Banach spaces which are also preduals of $M^{\alpha,s}_p$.   The first one is of course the K\"othe dual space $(M^{\alpha,s}_p)'$ defined  earlier in \eqref{KothDef}. The second one is a block type space in the spirit of \cite{BRV}, which we call 
$B^{\alpha,s}_q$, $q>1$.


\begin{definition} Let $q>1$, $\alpha>0$, and $s>1$. We define  $B^{\alpha,s}_q=B^{\alpha,s}_q(\RR^n)$ to be the space of all functions $f$ of the form 
	\begin{align*}
	f=\sum_{j}c_{j}a_{j},
	\end{align*}
	where the convergence is  in pointwise a.e. sense. Here $\{c_{j}\}\in l^{1}$ and each $a_j\in L^{q}(\RR^n)$ is such that there exists a bounded set $A_j\subset\RR^n$
	for which $a_{j}=0$ a.e. in $\RR^n\setminus A_j$ and  $	\|a_{j}\|_{L^{q}}\leq \emph{Cap}_{\alpha,s}(A_{j})^{\frac{1-q}{q}}$.
	 The norm of a function $f\in B^{\alpha,s}_q$ is defined as 
	\begin{align*}
	\|f\|_{B^{\alpha,s}_q}=\inf\Big\{\sum_{j}|c_{j}|:  f=\sum_{j}c_{j}a_{j} {\rm ~a.e.}\Big\}.
	\end{align*} 	
\end{definition}

It is now easy to see  from the definition that  $B^{\alpha,s}_{q}$ is a Banach space. Both $(M^{\alpha,s}_p)'$ and $B^{\alpha,s}_{p'}$ are also  preduals of $M^{\alpha,s}_p$.

\begin{theorem}\label{BspaceTh} Let $p>1$, $\alpha>0$, and $s>1$.
	We have   $$[(M^{\alpha,s}_p)']^*=\left(B^{\alpha,s}_{p'}\right)^{*}= M^{\alpha,s}_{p},$$ 
	with equalities of norms.
\end{theorem}

Having introduced several predual spaces to $M^{\alpha,s}_{p}$, a natural question to us now is whether they are isometrically isomorphic or at least isomorphic. We will show eventually that they are all indeed isomorphic. 
In the case  the capacity ${\rm Cap}_{\alpha, s}$ is strongly subadditive we claim that 
\begin{equation}\label{isometric}
N^{\alpha, s}_{p'} \approx (M^{\alpha, s}_p)' = \widetilde{N}^{\alpha, s}_{p'} =  B^{\alpha, s}_{p'},
\end{equation}
provided  ${\rm Cap}_{\alpha, s}$ is strongly subadditive. 

The first relation in \eqref{isometric} provides us with a new concrete description for the abstract space  
$(M^{\alpha, s}_p)'$ and enables us `to do harmonic analysis' on it when ${\rm Cap}_{\alpha, s}$ is strongly subadditive. In order to deal with all capacities,
we now introduce another space which we call   $\mathcal{N}^{\alpha, s}_{q}$, $q>1$. Eventually, we show that 
$\mathcal{N}_{p'}^{\alpha,s}  \approx (M_{p}^{\alpha,s})'$ for all  $p>1$ and $\alpha s\leq n$. To this end, we first modify the space  
$L^1(C)$, which  in general is only a quasinormed space. 
Let $\mathcal{L}^{1}(C)$ be the space of measurable functions $w$ such that 
$$\sup_{g}\int |g(x)| |w(x)| dx <+\infty, $$ 
where the supremum is taken over all $g\in M^{\alpha,s}_{1}$ such that  $\norm{g}_{M^{\alpha,s}_{1}}\leq 1$. In other words, $\mathcal{L}^{1}(C)$
is the K\"othe dual of $M^{\alpha,s}_{1}$ with the norm $\|w\|_{\mathcal{L}^{1}(C)}$ being  defined as the above supremum. It is easy to see that    $L^1(C) \hookrightarrow \mathcal{L}^1(C)$.

For $q>1$, we now   define $\mathcal{N}^{\alpha,s}_q=\mathcal{N}^{\alpha,s}_q(\mathbb{R}^n)$
as the space of all measurable functions $g$ such that there exists a weight $w\in \mathcal{L}^{1}(C)\cap A^{\rm loc}_1$ with $\|w\|_{\mathcal{L}^{1}(C)}\leq 1$ and $[w]_{A^{\rm loc}_1}\leq {\bf \overline{c}}(n,\alpha,s)$ such 
that 
$$\left(\int_{\mathbb{R}^n}|g(x)|^{q}w(x)^{1-q}dx\right)^{1/q}<+\infty.$$

As in the case of $N^{\alpha,s}_q$, the norm of a function $g\in \mathcal{N}^{\alpha,s}_q$ is the defined as the infimum of the left-hand side above
over all $w\in \mathcal{L}^{1}(C)\cap A^{\rm loc}_1$ with $\|w\|_{\mathcal{L}^{1}(C)}\leq 1$ and $[w]_{A^{\rm loc}_1}\leq {\bf \overline{c}}(n,\alpha,s)$.

\begin{theorem}\label{isomorphism} Let $p>1$, $\alpha>0$, $s>1$, $\alpha s\leq n$. Then $\mathcal{N}^{\alpha, s}_{p'}$  and $(M^{\alpha, s}_p)'$ are Banach function spaces (see Sub-section \ref{BFS}), and $\mathcal{N}^{\alpha, s}_{p'} \approx (M^{\alpha, s}_p)'$ (thus $(\mathcal{N}^{\alpha, s}_{p'})^* \approx M_p^{\alpha,s}$).	Moreover, if 
 ${\rm Cap}_{\alpha, s}$ is strongly subadditive then  $N^{\alpha, s}_{p'}$,  $\widetilde{N}^{\alpha, s}_{p'}$, and $B^{\alpha, s}_{p'}$ are also Banach function spaces, and
\begin{equation*}
\mathcal{N}^{\alpha, s}_{p'} \approx  N^{\alpha, s}_{p'} \approx (M^{\alpha, s}_p)' = \widetilde{N}^{\alpha, s}_{p'} =  B^{\alpha, s}_{p'}.
\end{equation*}
\end{theorem}


Finally, we have the following  isomorphism result which applies to all capacities.

\begin{theorem}\label{isomorphism2} Let $p>1$, $\alpha>0$, $s>1$, $\alpha s\leq n$. We have 
	\begin{equation}\label{chain}
	\mathcal{N}^{\alpha, s}_{p'} \approx  N^{\alpha, s}_{p'} \approx (M^{\alpha, s}_p)' \approx \widetilde{N}^{\alpha, s}_{p'} \approx  B^{\alpha, s}_{p'}.
	\end{equation}
\end{theorem}

In general, the space of continuous functions with compact support  $C_c$ is not dense in $M^{\alpha,s}_p$. We shall let 
 $\mathring{M}_{p}^{\alpha,s}$ denote the closure of $C_c$ in $M^{\alpha,s}_{p}$.
  As it turns out,  we have that $\mathring{M}_{p}^{\alpha,s}$ is a predual of $\mathcal{N}^{\alpha,s}_{p'}$. 
\begin{theorem}\label{triplettheorem}
	Let $p>1$, $\alpha>0$, $s>1$, with $\alpha s\leq n$.
We have 
$$(\mathring{M}^{\alpha,s}_{p})^{*} \approx \mathcal{N}^{\alpha,s}_{p'}$$ in the sense that each bounded linear functional $L\in (\mathring{M}^{\alpha,s}_{p})^{*}$ corresponds to a unique 
$g\in \mathcal{N}^{\alpha,s}_{p'}$ such that $L(v)=\int_{{\RR^n}}v(x)g(x)dx$ for all $v\in \mathring{M}^{\alpha,s}_{p}$, and  $\|g\|_{\mathcal{N}^{\alpha,s}_{p'}}\simeq \|L\|_{(\mathring{M}^{\alpha,s}_{p})^{*}}$.
\end{theorem}


As a consequence of Theorem \ref{triplettheorem}, we obtain a triplet duality   relation 
$$\mathring{M}_{p}^{\alpha,s}\text{--}\mathcal{N}_{p'}^{\alpha,s}\text{--}M_{p}^{\alpha,s},$$ 
which is analogous to the famous triplet $VMO\text{--}H^{1}\text{--}BMO$ of  harmonic analysis (see \cite{CW}). 
See also \cite{AX2} where a similar triplet was claimed without proof in the context of Morrey spaces. We mention that our proof of Theorem 
\ref{triplettheorem}  is completely different from the $VMO\text{--}H^1$ duality proof of \cite{CW}. It is based on the relation $\mathcal{N}^{\alpha, s}_{p'} \approx (M^{\alpha, s}_p)'$,  Radon-Nikodym Theorem, and Hahn-Banach Theorem. Moreover, it  can also  be easily modified to   provide a proof the claimed triplet in \cite{AX2}. For other  related results 
 in the Morrey space setting, see \cite{ST, ISY}.

Thanks to the way the spaces $N^{\alpha, s}_{p'}$ and $\mathcal{N}^{\alpha, s}_{p'}$ are constructed and  Theorem \ref{isomorphism2}, we obtain the following important results regarding the behavior of the Hardy-Littlewood maximal functions 
 and Calder\'on-Zygmund operators on those spaces.

\begin{theorem}\label{Mlocbound} Let $p>1$, $\alpha>0$, $s>1$, and $\alpha s\leq n$. Then the local  Hardy-Littlewood maximal function 
	${\bf M}^{\rm loc}$  is bounded on $S$ where $S$ is any of the spaces in \eqref{chain}.
\end{theorem}

We  recall the Hardy-Littlewood maximal function ${\bf M}$ is bounded on  
$M_{p}^{\alpha ,s}$, $\alpha s\leq n$, (see \cite{MS1}). However, unlike ${\bf M}$, standard singular integrals are generally unbounded on 
$M_{p}^{\alpha ,s}$. Take for example the $j$-th Riesz transform, 
$$R_j(f)(x)=c(n)\, {\rm p.v.}\int \frac{x_j-y_j}{|x-y|^{n+1}} f(y)dy , \qquad j=1,2,\dots,n,$$
and adapt the argument of  \cite[Theorem 1.1]{RT} to our setting, using the fact that $L^\infty\hookrightarrow M_{p}^{\alpha ,s}$.

On the other hand, ${\bf M}$ fails to be bounded on any of the spaces in \eqref{chain}, since they are included in $L^1$. Likewise, the first Riesz transform $R_1$, say, is also  unbounded on these spaces.
To see that, take a nonnegative function $f\in C_c^\infty(B_1(0))$ such that $f=1$ on $B_{1/2}(0)$.  Then for any $x=(x_1,x')=(x_1, x_2,\dots, x_n)$ with $x_1>1$ we have 
$$R_1(f)(x) \geq c(n)\int_{B_{1/2}(0)} \frac{x_1-y_1}{|x-y|^{n+1}} dy\geq c\, \frac{x_1}{|x|^{n+1}}.$$ 

This shows that $R_1(f)\not\in L^1$, since 
\begin{align*}
\|R_1(f)\|_{L^1} &\geq c \int_1^\infty \int_{|x'|<x_1} \frac{x_1}{|x|^{n+1}} dx' dx_1 \geq c \int_1^\infty \int_{|x'|<x_1} x_1^{-n} dx' dx_1\\
&=c \int_1^\infty x_1^{-1} dx_1 =+\infty,
\end{align*}
and thus it does not belong to any of the mentioned spaces.

However, the following `localized' boundedness property is applicable to  ${\bf M}$ and any standard Calder\'on-Zygmund  operator.

\begin{theorem}\label{CZop} Let $q>1$, $\alpha>0$, $s>1$, and $\alpha s\leq n$. Suppose that $T$ is an operator (not necessarily linear or sublinear) such that 
	$$\int |T (f)|^{q} w dx \leq C_1 \int |f|^{q} w dx$$
holds for all $f\in L^{q}(w)$ and all $w\in A_1$, with a constant $C_1$ depending only on $n, q$, and the bound for the $A_1$ constant of $w$.	
Then for any measurable function $f$ such that ${\rm supp}(f)\subset B_{R_0}(x_0)$,  $x_0\in\RR^n, R_0>0$, we have 
$$\|T(f) \chi_{B_{R_0}(x_0)}\|_{S}\leq C_2 \|f \|_{S},$$
where $S= N^{\alpha, s}_{q}, \mathcal{N}^{\alpha, s}_{q}$, $(M_{q'}^{\alpha ,s})'$, $\widetilde{N}^{\alpha, s}_{q}$, $B_{q}^{\alpha, s}$, or $M_{q}^{\alpha ,s}$. Here the constant  $C_2= C_2(n, \alpha, s, q, R_0)$.
\end{theorem}

We mention that  Theorem \ref{CZop} can be applied to the so-called (nonlinear) $m$-harmonic transform $\mathcal{H}_m$, $m>1$, where for each vector field $F\in L^m(\Omega, \RR^n)$ we define 
 $\mathcal{H}_m(F)=\nabla u$ with $u\in W^{1,m}_0(\Omega)$ being the unique solution of $\Delta_m u= {\rm div} (|F|^{m-2} F)$ in $\Omega$. Here $\Omega$ is a bounded $C^1$ domain in $\RR^n$ and $\Delta_m$ is the $m$-Laplacian defined as
$\Delta_m u ={\rm div} (|\nabla u|^{m-2} \nabla u)$. Indeed, this is possible since the weighted bound 
$$\int_{\Om} |\mathcal{H}_m(F)|^q w dx \leq C(n, m ,q, \Omega,  [w]_{A_1}) \int_{\Om} |F|^{q} w dx $$ 
holds for all weights $w\in A_1$ and $q\geq m$ (see \cite{Ph1, MP} for $q>m$ and \cite{AP1} for $q=m$). For $m$-Laplace equations with measure data, where the exponent $q$ can be less than the natural exponent $m$, see \cite{Ph3, NP}.

In Section \ref{homogeneoussetting} below, we shall discuss about the homogeneous versions of Theorems \ref{Mlocbound} and \ref{CZop} which involve Riesz potentials and Riesz capacities. We mention here that  results in the homogeneous setting are neater as they require no localization. 

\vspace{.2in}
\noindent {\bf Notation.} In the above and in what follows, for two quasinormed spaces $F$ and $G$ we write $F\approx G$ (respectively, $F=G$) to indicate that the two spaces are isomorphic (respectively, isometrically isomorphic). For two quantities $A$ and $B$, we write $A\simeq B$ to mean that there exist positive constants $c_1$ and $c_2$ such that $c_1 A\leq B\leq c_2 A$.

\section{Preliminaries}\label{Pre}

\subsection{Capacities and the space $L^1(C)$}
Recall that the Bessel capacity ${\rm Cap}_{\alpha, s}(\cdot)$, $\alpha>0, s>1$, is defined for every subset $E$ of $\RR^n$ by \eqref{BC}. It is an outer capacity, i.e., for any 
set $E\subset\RR^n$, 
$${\rm Cap}_{\alpha, s}(E)=\inf\{ {\rm Cap}_{\alpha, s}(G): G \supset E, G \text{ open}\},$$
and is countably subadditive in the sense that for ${E_i}\subset\RR^n$, $i=1,2, \dots,$ $$\text{Cap}_{\alpha,s}\Big(\bigcup_{i=1}^{\infty}E_{i}\Big)\leq\sum_{i=1}^{\infty}\text{Cap}_{\alpha,s}(E_{i}).$$

Moreover, it has the following basic properties of a Choquet capacity (see \cite{AH}):

 {\rm (i)}   $ \text{Cap}_{\alpha,s}(\emptyset)=0;$

 {\rm (ii)}  if $E_{1}\subset E_{2},$ then   $\text{Cap}_{\alpha,s}(E_{1})\leq\text{Cap}_{\alpha,s}(E_{2});$

 {\rm (iii)} if $K_1\supset K_2\supset \dots$ is a decreasing sequence of compact sets of $\RR^n$, then
$$\text{Cap}_{\alpha,s}\Big(\bigcap_{i=1}^{\infty}K_{i}\Big)=\lim_{i\rightarrow\infty}\text{Cap}_{\alpha,s}(K_{i});$$

 {\rm (iv)} if $E_1\subset E_2\subset \dots$ is an increasing sequence of  subsets of $\RR^n$, then
\begin{equation}\label{inc}
\text{Cap}_{\alpha,s}\Big(\bigcup_{i=1}^{\infty}E_{i}\Big)=\lim_{i\rightarrow\infty}\text{Cap}_{\alpha,s}(E_{i}).
\end{equation}

Thus by the Capacitability Theorem  (see \cite{Cho, Mey}), for any Borel (or more generally Suslin) set $E\subset\RR^n$ we have  
$$\text{Cap}_{\alpha,s}(E)=\sup\{\text{Cap}_{\alpha,s}(K): K\subset E, K \text{ compact}\}.$$

By \eqref{WH} we see that if $\alpha$ is a positive integer, then 
${\rm Cap}_{\alpha,s}(E)\simeq  C_{\alpha,s}(E)$ for any set $E\subset\RR^n$ (see also \cite{AH}). Here for a compact set $K\subset\RR^n$ and $\alpha\in \NN$, we define
$$C_{\alpha, s}(K)=\inf\{  \|\varphi\|_{W^{\alpha ,s}}: \varphi\in C_c^\infty, \varphi\geq 1 \text{~on~} K \},$$
and $C_{\alpha, s}(\cdot)$ is extended to any set $E$ of $\RR^n$ by letting 
\begin{equation}\label{extension}
C_{\alpha, s}(E):=\inf_{\substack{G\supset E\\ G\, {\rm open} }}\left\{ \sup_{\substack{K\subset G \\ K \, {\rm compact}}} C_{\alpha, s}(K) \right\}.
\end{equation}

For $s=2$ and $\alpha\in (0,1]$, it is known that  ${\rm Cap}_{\alpha,s}(\cdot)$ is strongly subadditive in the sense of \eqref{StrongS}
 (see \cite[pp. 141--145]{Lan}). We note that the book \cite{Lan} considers only  Riesz capacities, i.e.,  homogeneous versions of  ${\rm Cap}_{\alpha,2}(\cdot)$.
 However, the argument there also applies to Bessel capacities    since for any $\alpha\in (0,1]$ the Bessel kernel $G_{2\alpha}$ is continuous and subharmonic in 
 $\RR^n\setminus \{0\}$ (hence the First Maximum Principle in the sense of  \cite[Theorem 1.10]{Lan} holds).
 
 On the other hand, for $\alpha=1$, the capacity $C_{1, s}(\cdot)$ is strongly subadditive for any $s>1$. Indeed, this can be  proved  by adapting the proof of \cite[Theorem 2.2]{HKM} to our nonhomogeneous setting.

We shall need the following metric properties of ${\rm Cap}_{\alpha,s}(\cdot)$ (see \cite{AH}): For any $0<r\leq 1$,
\begin{equation}\label{rleq1}
{\rm Cap}_{\alpha,s}(B_r)\simeq r^{n-\alpha s} \quad \text{if } \alpha s<n
\end{equation}

and 
$${\rm Cap}_{\alpha,s}(B_r)\simeq [\log(\tfrac{2}{r})]^{1-s} \quad \text{if } \alpha s=n.$$

For $r\geq 1$ and $\alpha s\leq n$ we have 
\begin{equation}\label{rgeq1}
{\rm Cap}_{\alpha,s}(B_r)\simeq r^n.
\end{equation}

On the other hand, we have for any non-empty set $E$ with ${\rm diam}(E)\leq 1$, 
\begin{equation}\label{alphasn}
{\rm Cap}_{\alpha,s}(E)\simeq 1 \quad \text{if } \alpha s>n.
\end{equation}

By Sobolev Embedding Theorem for any Lebesgue measurable set $E$,
\begin{equation}\label{aslessn}
|E|^{1-\alpha s/n} \leq C\, {\rm Cap}_{\alpha,s}(E) \quad \text{if } \alpha s <n.
\end{equation}

Moreover, by Young's inequality for convolution we have, for $s=n/\alpha>1$, 
$$\|G_\alpha*f\|_{L^q}\leq \|G_\alpha\|_{L^r} \|f\|_{L^{\frac{n}{\alpha}}}$$
for any $n/\alpha\leq q<+\infty$ and $r=nq/(n+q(n-\alpha))$. Thus for any $\epsilon\in (0,1]$ we find
\begin{equation}\label{asequaln}
|E|^\epsilon \leq C(\epsilon)\, {\rm Cap}_{\alpha,s}(E) \quad \text{if } \alpha s =n.
\end{equation}

Note that using the bound $\|G_\alpha*f\|_{L^s}\leq \|G_\alpha\|_{L^1} \|f\|_{L^{s}},$
we also find that
\begin{equation*}
|E| \leq C\, {\rm Cap}_{\alpha,s}(E) \quad \text{for all } \alpha>0, s>1.
\end{equation*}

It follows  that if ${\rm Cap}_{\alpha,s}(E)=0$ then  the Lebesgue measure of $E$ is zero. Moreover, if $f\in L^\infty(\RR^n)$ then $f\in M^{\alpha, s}_p$ for any $p\geq 1$, and 
\begin{equation}\label{LinftyM}
\|f\|_{M^{\alpha, s}_p}\leq C \|f\|_{L^\infty(\RR^n)}.
\end{equation}

On the other hand, when $\alpha s<n$ by \eqref{aslessn} we have  
\begin{equation}\label{LweakM}
\|f\|_{M^{\alpha, s}_p}\leq C \|f\|_{L^{\frac{n p}{\alpha s},\infty}(\RR^n)},
\end{equation}
where $L^{\frac{np}{\alpha s},\infty}(\RR^n)$ is  the weak $L^{\frac{np}{\alpha s}}$ space.

\begin{remark} Let $S$ be any of the spaces  in \eqref{chain}.  As $S^*\approx M^{\alpha,s}_{p}$, using the embedding \eqref{LinftyM},  we see that 
	$ S \hookrightarrow L^1$, $\alpha s\leq n$. Likewise,  by \eqref{LweakM} we also have   $S \hookrightarrow L^{\frac{np}{np- \alpha s }, 1}$,
	$\alpha s<n$.
	Here $L^{\frac{np}{np-\alpha s}, 1}$ is a Lorentz space which is the predual of  $L^{\frac{np}{\alpha s}, \infty}$.
\end{remark}

%

The Choquet integral  of a $\text{Cap}_{\alpha,s}$-quasieverywhere defined function $w:\RR^n \rightarrow [0,\infty]$   was defined by
\eqref{ChoI}. We  also let  $L^1(C)$ be the space of quasicontinuous  functions $f$ in $\mathbb{R}^n$ such that  \eqref{L1C} holds. Perhaps, a better notation for 
$L^1(C)$ should be $L^1({\rm Cap}_{\alpha,s})$ to indicate its dependence on ${\rm Cap}_{\alpha,s}$. But we shall use the notation   $L^1(C)$ for simplicity and implicitly understand that $C={\rm Cap}_{\alpha,s}$.

In general,
the `norm' of $L^1(C)$ is only a quasinorm, i.e., we only  have
$$\| f+g\|_{L^1(C)}\leq 2 \|g\|_{L^1(C)}+ 2 \|g\|_{L^1(C)}. $$
However, if  $\text{Cap}_{\alpha,s}$ is strongly subadditive then it is actually a norm by a theorem of Choquet (see \cite{Cho, Den}).

In \cite[Theorem 4]{Ad3} the following quasiadditivity result was obtained for $\text{Cap}_{\alpha,s}$:
\begin{equation}\label{QAC}
\sum_{j=1}^{\infty}\text{Cap}_{\alpha,s}(E\cap\{j-1\leq|x|<j\})\leq C \, \text{Cap}_{\alpha,s}(E)
\end{equation}
for all $E\subset\RR^n$, where $C=C(n,\alpha, s)>0$. We now use \eqref{QAC} to obtain the following density result for the space  $L^1(C)$.
\begin{proposition}\label{L1}
	$C_{c}(\RR^n)$ is dense in $L^{1}(C)$, where $C_{c}(\RR^n)$ is the linear space of
continuous functions with compact support in $\RR^n$.	 
\end{proposition}

\begin{proof}
	We first show that the set of all bounded continuous functions is dense in $L^{1}(C)$. Let $f\in L^{1}(C)$ be given. For $M>0$, we define $f_{M}(x)=f(x)$ if $|f(x)|\leq M$, $f_M(x)=M$ if $f(x)>M$, and $f_M(x)=-M$ if $f(x)<-M$.  Note that 
	\begin{align*}
	\|f_{M}-f\|_{L^{1}(C)}&=\int_{0}^{\infty}\text{Cap}_{\alpha,s}(\{|f_{M}-f|>t\})dt\\
	&=\int_{0}^{\infty}\text{Cap}_{\alpha,s}(\{|f|>M+t\})dt\\
	&=\int_{M}^{\infty}\text{Cap}_{\alpha,s}(\{|f|>t\})dt \rightarrow 0,
	\end{align*}
	as $M\rightarrow\infty$. For any $\epsilon>0$, choose an $M>0$ such that $\|f_{M}-f\|_{L^{1}(C)}<\epsilon$. As $f_{M}$ is quasicontinuous (since $f$ is quasicontinuous), there exists an open set $G$ such that $\text{Cap}_{\alpha,s}(G)<\epsilon$ and $f_{M}\big|_{G^{c}}$ is continuous. 
	
	By Tietze Extension Theorem, we can find a continuous function $v$ such that $|v|\leq M$ and $v=f_{M}$ on $G^{c}$. Then 
	\begin{align*}
	\|v-f_{M}\|_{L^{1}(C)}&=\int_{0}^{\infty}\text{Cap}_{\alpha,s}(\{|v-f_{M}|>t\})dt\\
	&=\int_{0}^{\infty}\text{Cap}_{\alpha,s}(\{x\in G: |v(x)-f_{M}(x)|>t\})dt\\
	&=\int_{0}^{2M}\text{Cap}_{\alpha,s}(\{x\in G: |v(x)-f_{M}(x)|>t\})dt\\
	&\leq 2M\text{Cap}_{\alpha,s}(G) <2M \, \epsilon.
	\end{align*}
	
	As a result, 
	$$\|f-v\|_{L^{1}(C)}\leq 2\|f-f_{M}\|_{L^{1}(C)}+2\|f_{M}-v\|_{L^{1}(C)}<2\epsilon+4M\, \epsilon,$$ which yields the claim.

	Now we claim that $C_{c}$ is dense in $L^{1}(C)$. All we need to do is to approximate bounded continuous functions by functions in $C_{c}$. To this end, let $v$ be a bounded continuous function, say, $|v|\leq M$ for some $M>0$. For each $N=1,2,...$, let $O_{N}=\{|v|>1/N\}$, then $O_{N}$ is open and $$\text{Cap}_{\alpha,s}(O_{N})\leq N \|v\|_{L^{1}(C)}<+\infty.$$

	We observe that  
	\begin{align*}
	\|v\chi_{O_{N}^{c}}\|_{L^{1}(C)}&=\int_{0}^{\infty}\text{Cap}_{\alpha,s}(\{x\in O_{N}^{c}: |v(x)|>t\})dt\\
	&=\int_{0}^{1/N}\text{Cap}_{\alpha,s}(\{x\in O_{N}^{c}: |v(x)|>t\})dt\\
	&\leq\int_{0}^{1/N}\text{Cap}_{\alpha,s}(\{|v|>t\})dt \rightarrow 0,
	\end{align*}
	as $N\rightarrow\infty$. Thus for any $\epsilon>0$, there is an open set $O$ such that $\text{Cap}_{\alpha,s}(O)<\infty$ and $\|v\chi_{O^{c}}\|_{L^{1}(C)}<\epsilon$. Since $\text{Cap}_{\alpha,s}(O)<\infty$, by \eqref{QAC} we have 
	\begin{align*}
	\sum_{j=0}^{\infty}\text{Cap}_{\alpha,s}(O\cap\{j\leq|x|<j+1\})<\infty,
	\end{align*}
	and so there is a positive integer $j_{0}$ such that 
	\begin{align*}
	\text{Cap}_{\alpha,s}(O\cap\{|x|\geq j_{0}\})\leq\sum_{j=j_{0}}^{\infty}\text{Cap}_{\alpha,s}(O\cap\{j\leq|x|<j+1\})<\epsilon.
	\end{align*}

	Let $O_{1}=O\cap\{|x|<j_{0}\}$ and $O_{2}=O\cap\{|x|\geq j_{0}\}$, then $O=O_{1}\cup O_{2}$, $O_{1}$ is bounded, and $\text{Cap}_{\alpha,s}(O_{2})<\epsilon$. Let $\eta$ be a continuous function with compact support such that $0\leq\eta\leq 1$ and $\eta \equiv1$ on $O_{1}$. We have 
	\begin{align*}
	\|\eta v-v\|_{L^{1}(C)}&\leq 2 \|(\eta v-v)\chi_{O^{c}}\|_{L^{1}(C)}\\
	&\quad + 4\|(\eta v-v)\chi_{O_{1}}\|_{L^{1}(C)}+ 4 \|(\eta v-v)\chi_{O_{2}}\|_{L^{1}(C)}\\
	&= 2\|(\eta v-v)\chi_{O^{c}}\|_{L^{1}(C)}+ 4\|(\eta v-v)\chi_{O_{2}}\|_{L^{1}(C)},
	\end{align*}
	since $\eta \equiv 1$ on $O_{1}$.

	On the other hand, note that 
	\begin{align*}
	\|(\eta v-v)\chi_{O^{c}}\|_{L^{1}(C)}&=\int_{0}^{\infty}\text{Cap}_{\alpha,s}(\{x\in O^{c}: |(\eta v)(x)-v(x)|>t\})dt\\
	&\leq \int_{0}^{\infty}\text{Cap}_{\alpha,s}(\{x\in O^{c}: 2|v(x)|>t\})dt\\
	&\leq 2\|v\chi_{O^c}\|_{L^{1}(C)}\\
	&<2\epsilon.
	\end{align*}
	
	Also, 
	\begin{align*}
	\|(\eta v-v)\chi_{O_{2}}\|_{L^{1}(C)}&\leq\int_{0}^{2M}\text{Cap}_{\alpha,s}(\{x\in O_{2}:|(\eta v)(x)-v(x)|>t\})dt\\
	&\leq 2M\text{Cap}_{\alpha,s}(O_{2})\\
	&<2M\, \epsilon.
	\end{align*}

	Thus, we conclude that 
	\begin{align*}
	\|\eta v-v\|_{L^{1}(C)}<4\epsilon+8M\, \epsilon,
	\end{align*}
and since $\eta v$ has compact support, the proof is then complete.
\end{proof}

We are now ready to establish the completeness of $L^1(C)$. 

\begin{proposition}\label{completeness} The quasinorm space	$L^{1}(C)$ is complete for any $\alpha>0$ and $s>1$.
\end{proposition}

\begin{proof}
	Let $\{u_{n}\}$ be a Cauchy sequence in $L^{1}(C)$. We need to show that  $u_{n}\rightarrow u$ in $L^1(C)$ for some $u\in L^{1}(C)$.
	Since $C_{c}$ is dense in $L^{1}(C)$, we may assume that $\{u_{n}\} \subset C_{c}$.

	As $\{u_{n}\}$ is a Cauchy sequence, we can find positive integers $n_{1}<n_{2}<\cdots$ such that 
	\begin{align}\label{4j}
	\int_{0}^{\infty}\text{Cap}_{\alpha,s}(\{|u_{m}-u_{n}|>t\})dt<4^{-j}
	\end{align}
	for all $m,n\geq n_{j}$, $j=1,2, \dots$ In particular,
	\begin{align*}
	\int_{0}^{2^{-j}} \text{Cap}_{\alpha,s}(|u_{n_{j+1}}-u_{n_{j}}|>2^{-j}) dt<4^{-j}
	\end{align*}
	and hence 
	\begin{align*}
	\text{Cap}_{\alpha,s}(|u_{n_{j+1}}-u_{n_{j}}|>2^{-j})<2^{-j}.
	\end{align*}

	Let $G_{j}=\{|u_{n_{j+1}}-u_{n_{j}}|>2^{-j}\}$. Then $G_{j}$ is open and $\text{Cap}_{\alpha,s}(G_{j})<2^{-j}$. We now set
	\begin{align*}
	H_{m}=\bigcup_{j\geq m}G_{j}.
	\end{align*}
	Then we have  
	\begin{align}\label{smallcap}
	\text{Cap}_{\alpha,s}(H_{m})\leq\sum_{j\geq m}\text{Cap}_{\alpha,s}(G_{j})<\sum_{j\geq m}2^{-j}\rightarrow 0
	\end{align}
	as $m\rightarrow\infty$.

	Observe that for any $x\in H_{m}^{c}$, we have
	\begin{align*}
	\sum_{j\geq m}|u_{n_{j+1}}(x)-u_{n_{j}}(x)|\leq\sum_{j\geq m}2^{-j}<+ \infty.
	\end{align*} 
	Thus if we let $u: H_{m}^{c}\rightarrow \RR$ be defined by 
	 \begin{align*}
	 u(x)=\lim_{k\rightarrow\infty}u_{n_{k}}(x)= u_{n_{m}}(x)+ \lim_{k\rightarrow \infty}\sum_{j=m+1}^k(u_{n_{j}}(x)-u_{n_{j-1}}(x)), 
	 \end{align*}
	then   by the  Weierstrass M-Test we see that   $u$ is continuous in $H_{m}^{c}$. 
	
	As the set $H_{m}^{c}$ is increasing, the function $u$ can be extended to define in 
	the union $\bigcup_{m\geq 1} H_m^c$. It is now easy to see from \eqref{smallcap} that $u$ is 
	 quasicontinuous.

	 Now by \eqref{inc} and the Monotone Convergence Theorem we have for each $n\geq 1$,  
	\begin{align*}
	\|u_{n}-u\|_{L^{1}(C)}&=\int_{0}^{\infty}\text{Cap}_{\alpha,s}(\{|u_{n}-u|>t\})dt\\
	&\leq\int_{0}^{\infty}\text{Cap}_{\alpha,s}\left(\bigcup_{N\geq 1}\bigcap_{k\geq N}\{|u_{n}-u_{n_{k}}|>t\}\right)dt\\
	&=\int_{0}^{\infty}\lim_{N\rightarrow\infty}\text{Cap}_{\alpha,s}\left(\bigcap_{k\geq N}\{|u_{n}-u_{n_{k}}|>t\}\right)dt\\
	&=\lim_{N\rightarrow\infty}\int_{0}^{\infty}\text{Cap}_{\alpha,s}\left(\bigcap_{k\geq N}\{|u_{n}-u_{n_{k}}|>t\}\right)dt.
	\end{align*} 
	
	Thus by \eqref{4j} for each $j=1,2, \dots$, and $n\geq n_j$, we have 
	\begin{align*}
	\|u_{n}-u\|_{L^{1}(C)} \leq 4^{-j}.
	\end{align*} 
	
	This completes the proof of the proposition.
\end{proof}

The following duality relation was stated without proof in \cite{Ad4}. Indeed, it can be  proved using  Proposition \ref{L1} and the formula 
$$\int_{\RR^n} u d|\mu|= \sup\left \{\int_{\RR^n} v d\mu: v\in C_c(\RR^n), |v|\leq u\right\}, $$
which holds for all $u\in C_c(\RR^n)$ and $u\geq 0$.

\begin{theorem}\label{first dual} Let $\alpha>0$ and $s>1$. We have 
	$\left(L^{1}(C)\right)^{\ast}=\mathfrak{M}^{\alpha,s}$ in the sense that each  bounded linear functional  $L  \in\left(L^{1}(C)\right)^{\ast}$ 
	corresponds to a unique measure $\nu\in \mathfrak{M}^{\alpha,s}$ 
	in such a way  that
	\begin{align}\label{inrep}
	L(f)=\int_{\RR^n}f(x)d\nu(x)
	\end{align}
	for all $f\in L^1(C)$.  Moreover, $\|L\|=\|\nu\|_{\mathfrak{M}^{\alpha,s}}$.
\end{theorem}

\begin{remark} The right-hand side of \eqref{inrep} makes sense since for  $f\in L^1(C)$ and $t\in \RR$, we have that  the set $\{ f>t\}=F\setminus N$ for a $G_\delta$ set $F$ 	and a set $N$ with $\mu(N)={\rm Cap}_{\alpha ,s}(N)=0$. Here $\mu$ should be understood as the completion of $\mu$, and note that if 
${\rm Cap}_{\alpha ,s}(N)=0$ then $N\subset \widetilde{N}$, where  $\widetilde{N}$ is a $G_\delta$ set with ${\rm Cap}_{\alpha ,s}(\widetilde{N})=0$.
	
\end{remark}

\subsection{Banach function spaces}\label{BFS} Most of the spaces under our consideration fit well in the context of Banach function spaces in the sense of \cite{Lux}. In the setting
of $\RR^n$ with Lebesgue measure as the underlying measure, a Banach function space $X$ on $\RR^n$ is the set of all Lebesgue measurable functions $f$ such that 
$\|f\|_{X}:=\rho(|f|)$ is finite. Here $\rho(f)$, $f\geq 0$, is a given metric function ($0\leq \rho(f)\leq \infty$) that obeys the following properties:

\vspace{.1in}
(P1) $\rho(f)=0$ if and only if $f(x)=0$ a.e. in $\RR^n$; $\rho(f_1+f_2)\leq \rho(f_1)+ \rho(f_2)$; and  $\rho( \lambda f)= \lambda \rho(f)$ for any constant $\lambda\geq 0$.

(P2) If $\{f_j\}$, $j=1,2, \dots$, is a sequence of nonnegative measurable functions and $f_j \uparrow f$ a.e. in $\RR^n$, then $\rho(f_j) \uparrow \rho(f).$

(P3) If   $E$ is any bounded  and measurable subset of $\RR^n$, and $\chi_E$ is its characteristic
	function, then $\rho(\chi_E)<+\infty$.

(P4) For every bounded and measurable subset  $E$ of $\RR^n$, there exists a finite constant
$A_E\geq 0$ (depending only on the set $E$) such that $\int_{E} f dx \leq A_E \rho(f)$
for any nonnegative measurable function $f$ in $\RR^n$.	

\vspace{.1in}

It follows from property (P2) that any Banach function space $X$ is complete (see \cite{Lux}). We also have that, for measurable functions $f_1$ and $f_2$, if $|f_1|\leq |f_2|$ a.e. in $\RR^n$ and $f_2 \in X$, then it follows that $f_1\in X$ and $\|f_1\|_X \leq \|f_2\|_X$.

Given  a Banach function space $X$, the K\"othe dual space (or the associate space) to $X$, denoted by $X'$, is the set of all measurable functions $f$ such that $f g \in L^1(\RR^n)$ for all $g \in X$. It turns out that $X'$ is also a Banach function space with the associate metric function $\rho'(f)$, $f\geq 0$, defined by
$$\rho'(f):= \sup\left\{\int |fg| dx: g\in X, \, \|g\|_{X}\leq 1  \right\}.$$ 

By definition, the second associate space $X''$ to $X$ is given by $X''=(X')'$, i.e., $X''$ is the K\"othe dual space to $X'$.  The following theorems are important in the theory of Banach function spaces (see \cite{Lux}).

\begin{theorem}\label{XX''} Every Banach function space $X$ coincides
	with its second associate space $X''$, i.e., $X=X''$ with  equality of norms.	
\end{theorem}

\begin{theorem}\label{X*X'} $X^*=X'$ (isometrically) if and only if the space $X$ has an absolutely continuous norm. 
\end{theorem} 
 
Here we say that $X$ has an absolutely continuous norm if the following properties are satisfied for any $f\in X$:

(a) If $E$ is a bounded set of $\RR^n$ and $E_j$ are measurable subsets of $E$ such that $|E_j|\rightarrow 0$ as $j\rightarrow \infty$, the $\|f\chi_{E_j}\|_{X}\rightarrow0$ as $j\rightarrow \infty$.

(b) $ \|f \chi_{\RR^n\setminus B_j(0)} \|_{X} \rightarrow 0$ as $j\rightarrow \infty$.

It is known that  $X$ has an absolutely continuous norm
if  and only if any  sequence $f_j\in X$ such that $|f_j| \downarrow 0$
a.e. in $\RR^n$ has the property that $\| f_j\|_X \downarrow 0$ (see \cite[page 14]{Lux}).  
 
It is easy to see from the definition that the space $M^{\alpha,s}_p$, $\alpha>0, s>1, p>1$, is a Banach function space in $\RR^n$ and so is its K\"othe dual space $(M^{\alpha,s}_p)'$.

We will now follow an idea in \cite[Proposition 2.11]{KV} and use  the $p$-convexity of $M^{\alpha,s}_p$ to show that $(M^{\alpha,s}_p)'$ is actually a predual space of   $M^{\alpha,s}_p$.

\begin{proposition}\label{M'*} We have $[(M^{\alpha,s}_p)']^*=M^{\alpha,s}_p$ (isometrically) for any $\alpha>0, s>1, p>1$.
\end{proposition}

\begin{proof}  It is obvious that $M^{\alpha,s}_p$ is $p$-convex with $p$-convexity constant $1$, i.e., for every choice of $m$ functions $\{f_i\}_{i=1}^m$ in  
$M^{\alpha,s}_p$, we have 
\begin{equation}\label{p-convex}
\norm{\Big(\sum_{i=1}^m |f_i|^p\Big)^{\frac{1}{p}}}_{M^{\alpha,s}_p} \leq \Big( \sum_{i=1}^m \norm{f_i}_{M^{\alpha,s}_p}^p\Big)^{\frac{1}{p}}.
\end{equation}

Now using the fact that $\ell^{p'}((M^{\alpha,s}_p)^*)=\ell^{p}(M^{\alpha,s}_p)^*$ we have for any choice of $m$ functions $\{g_i\}_{i=1}^m$ in  
$(M^{\alpha,s}_p)'$,
\begin{align*}
\Big(\sum_{i=1}^m & \norm{g_i}^{p'}_{(M^{\alpha,s}_p)'} \Big)^{\frac{1}{p'}}=\Big(\sum_{i=1}^m \norm{g_i}^{p'}_{(M^{\alpha,s}_p)^*} \Big)^{\frac{1}{p'}}\\
& = \sup_{\norm{\{f_i\}}_{\ell^p(M^{\alpha,s}_p)}\leq 1} \sum_{i=1}^m \int f_i(x) g_i(x) dx\\
& \leq \sup_{\norm{\{f_i\}}_{\ell^p(M^{\alpha,s}_p)}\leq 1}  \int \Big(\sum_{i=1}^m |f_i(x)|^{p}\Big)^{\frac{1}{p}}  \Big(\sum_{i=1}^m |g_i(x)|^{p'}\Big)^{\frac{1}{p'}} dx\\
& \leq \sup_{\norm{\{f_i\}}_{\ell^p(M^{\alpha,s}_p)}\leq 1} \norm{\Big(\sum_{i=1}^m |f_i|^{p}\Big)^{\frac{1}{p}}}_{M^{\alpha,s}_p} \norm{\Big(\sum_{i=1}^m |g_i|^{p'}\Big)^{\frac{1}{p'}}}_{(M^{\alpha,s}_p)'}.
\end{align*}

Thus in view of \eqref{p-convex} we see that $(M^{\alpha,s}_p)'$ is $p'$-concave with $p'$-concavity constant $1$, i.e., 
\begin{align*}
\Big(\sum_{i=1}^m  \norm{g_i}^{p'}_{(M^{\alpha,s}_p)'} \Big)^{\frac{1}{p'}} \leq \norm{\Big(\sum_{i=1}^m |g_i|^{p'}\Big)^{\frac{1}{p'}}}_{(M^{\alpha,s}_p)'}.
\end{align*}

Then by \cite[Proposition 1.a.7]{LT} the space $(M^{\alpha,s}_p)'$ must have  an absolutely continuous norm. Hence, Theorems \ref{XX''} and \ref{X*X'}  yield that 
$$[(M^{\alpha,s}_p)']^*= (M^{\alpha,s}_p)''=M^{\alpha,s}_p,$$
as desired.
\end{proof}

\begin{remark} The proof shows that $(M^{\alpha,s}_p)'$ has an absolutely continuous norm and thus it is a separable Banach  space (see \cite{Lux}).
\end{remark}

\begin{remark} We notice that by \cite[Theorem 1.2]{MV} it holds that  
	$$\|f\|_{M^{\alpha,s}_p} \simeq \left\|\frac{G_\alpha*[G_\alpha*(|f|^p)]^{s'}}{G_\alpha*(|f|^p)}\right\|_{L^\infty(\{G_\alpha*(|f|^p)>0\})}^{\frac{1}{p(s'-1)}}.$$

Thus it follows from \cite[Proposition 2.9]{KV} and \cite[Theorem 2.10]{KV} that both $|f|^{\frac{p}{s'}}$ and $G_\alpha*(|f|^p)$ belong to $\mathcal{Z}$.
Here $\mathcal{Z}=\mathcal{Z}_{s'}$ is the space of measurable functions $h$ such that the integral equation 
$$u=G_\alpha*(u^{s'}) + \epsilon |h|\qquad \text{a.e.}$$
has a nonnegative solution $u\in L_{\rm loc}^{s'}(\RR^n)$ for some $\epsilon>0$.   A norm for $\mathcal{Z}$ can be defined by
$$\|h\|_{\mathcal{Z}}=\inf\{t>0: G_\alpha* |h|^{s'} \leq t^{s'-1} |h| \quad \text{a.e.} \}=\left\|\frac{G_\alpha*|h|^{s'}}{|h|}\right\|_{L^\infty(\{|h|>0\})}^{\frac{1}{s'-1}}$$
(see \cite[page 3455]{KV}). Moreover, by \cite[Theorem 2.10]{KV} we have 
$$\| |f|^{\frac{p}{s'}}\|_{\mathcal{Z}} \simeq \| G_\alpha*(|f|^{p})\|_{\mathcal{Z}}^{\frac{1}{s'}}.$$

With these observations, we see that a function $f\in M^{\alpha,s}_p$ if and only if $|f|^{\frac{p}{s'}}\in \mathcal{Z}$ and 
 	$$\left\| |f|^{\frac{p}{s'}}\right\|^{\frac{s'}{p}}_{\mathcal{Z}} \simeq \|f\|_{M^{\alpha,s}_p}.$$
 	
In particular, for $p=s'$, we have $\mathcal{Z}\approx M^{\alpha,s}_{s'}$ and   	$\mathcal{Z}' \approx (M^{\alpha,s}_{s'})'$. Thus by \cite[Theorem 2.12]{KV} it holds that
$$\|f\|_{(M^{\alpha,s}_{s'})'}\simeq \inf\Big\{ \int h^s (G_\alpha* h)^{1-s} dx: h\geq |f| \quad \text{a.e.} \Big\}.$$

However, this interesting equivalence will not be used in this paper.
\end{remark}

\section{Proof of Theorems \ref{firsttheorem}, \ref{second}, and \ref{NoA1Th}}

To prove Theorem \ref{firsttheorem} we need the following preliminary results. A homogeneous version of the next theorem can be found in \cite{MS1}. But our approach here is different   from that of \cite{MS1} at least in the case $s>2-\alpha/n$.

\begin{theorem}\label{mainA1loc} Let $s>1$, $\alpha>0$, and $\alpha s\leq n$. If $t\in (1, n/(n-\alpha))$ for $s\leq 2-\alpha/n$ and $t\in (1, n(s-1)/(n-\alpha s))$ 
	for $s>2-\alpha/n$, then  for any nonnegative measure $\mu$ and $V=G_\alpha*(G_\alpha*\mu)^{\frac{1}{s-1}}$ we have 
	\begin{equation}\label{AV}
	{\bf M}^{\rm loc}(V^t)(x_0)\leq A V^t(x_0), \qquad \forall x_0\in \RR^n,
	\end{equation}
	where $A$ is a constant independent of $\mu$. 	
\end{theorem}

\begin{proof} We shall use the following properties of $G_\alpha$ (see \cite[Sect. 1.2.4]{AH}):
	\begin{equation}\label{G1}
	G_\alpha (x) \simeq |x|^{\alpha-n}, \qquad \forall |x|\leq 4, \qquad  0<\alpha<n, 
	\end{equation}	
	and 
	\begin{equation}\label{G2}
	G_\alpha (x) \leq c\,  G_\alpha(x+y), \qquad \forall |x|\geq 2, \qquad |y|\leq 1. 
	\end{equation}		
	
	Note that \eqref{G1} and \eqref{G2} yield that for any $t\in (1, n/(n-\alpha))$ we have 
	\begin{equation}\label{AG}
	{\bf M}^{\rm loc}(G^t_\alpha(\cdot-z))(x)\leq C\, G^t_\alpha(x-z), \qquad \forall x, z\in \RR^n,
	\end{equation}
	where $C$ is independent of $x$ and $z$.  This can be verified by inspecting the case $x\in B_{3}(z)$ and the case $x\not \in B_{3}(z)$ separately.

	First we consider the case $s\leq 2-\alpha/n$ and $t\in (1, n/(n-\alpha))$. 
	By Minkowski's inequality and \eqref{AG} we have, for $x_0\in \RR^n$ and $r\in(0,1]$, 
	\begin{align*}
	\fint_{B_r(x_0)} V^t(y) dy &=\fint_{B_r(x_0)} \left[\int_{\RR^n} G_\alpha(y-z) (G_\alpha * \mu)(z)^{\frac{1}{s-1}} dz\right]^{t} dy\\
	& \leq  \left[\int_{\RR^n}  (G_\alpha * \mu)(z)^{\frac{1}{s-1}}  \left( \fint_{B_r(x_0)} G^t_\alpha(y-z) dy\right)^{\frac{1}{t}}  dz\right]^{t} \\
	& \leq C\, \left[\int_{\RR^n}  (G_\alpha * \mu)(z)^{\frac{1}{s-1}}  G_\alpha(x_0-z)  dz\right]^{t} \\
	&= C\, V^t(x_0).
	\end{align*}
	
	Thus we get \eqref{AV} when  $t\in (1, n/(n-\alpha))$ and $s\leq 2-\alpha/n$. In fact, the proof is valid for all $s>1$.	
	
	We now consider the case $s> 2-\alpha/n$ and $t\in (1, n(s-1)/(n-\alpha s))$.  By H\"older's inequality we may assume that $t>s-1$. 	
	Let $x_0\in \RR^n$ and $r\in(0,1]$.
	We write
	$$V(x)=V_1(x)+V_2(x),$$
	where 
	$$V_1(x)=\int_{|y-x_0|>3} G_\alpha(x-y)\varphi(y) dy,$$
	$$V_2(x)=\int_{|y-x_0|\leq 3} G_\alpha(x-y)\varphi(y) dy,$$	
	with
	$$\varphi(y)=(G_\alpha*\mu(y))^{\frac{1}{s-1}}.$$
	
	Observe that for $|x-x_0|\leq 1$ and $|y-x_0|>3$, it holds that 
	\begin{equation}\label{V1}
	G_\alpha(x-y) \leq A G_\alpha (x_0-y).
	\end{equation}
	
	Indeed, since $|x-y|\geq |y-x_0|-|x-x_0|\geq 3-1=2$ and $|x-x_0|\leq 1$, by \eqref{G2} we find
	$$G_\alpha(x-y) \leq A G_\alpha (x-y+x_0-x)=  A G_\alpha (x_0-y).$$

	Now by \eqref{V1} we have 
	$$V_1(x)\leq A \int_{\RR^n} G_\alpha(x_0-y)\varphi(y) dy =AV(x_0)$$
	for all $|x-x_0|\leq 1$. This yields
	\begin{equation}\label{E1}
	\fint_{B_r(x_0)} V^t_1(x) dx \leq A V^t(x_0).
	\end{equation}
	
	As for $V_2$ we write
	$$V_2(x)=\int_{|x-u-x_0|\leq 3} G_\alpha(u)\varphi(x-u) du\leq c [V_{21}(x) + V_{22}(x)],$$
	where 
	$$V_{21}(x)=\int_{|x-u-x_0|\leq 3} G_\alpha(u)\varphi_1(x-u) du,$$
	$$V_{22}(x)=\int_{|x-u-x_0|\leq 3} G_\alpha(u)\varphi_2(x-u) du,$$
	with 
	$$\varphi_1(x-u)=\left(\int_{|z-x_0|>5} G_\alpha(x-u-z) d \mu(z)\right)^{\frac{1}{s-1}},$$
	$$\varphi_2(x-u)=\left(\int_{|z-x_0|\leq 5} G_\alpha(x-u-z) d \mu(z)\right)^{\frac{1}{s-1}}.$$	
	
	Using \eqref{G2}, for $|z-x_0|>5$, $|x-u-x_0|\leq3$, and $|x-x_0|\leq 1$, we have 
	$$G_\alpha(x-u-z) \leq A G_\alpha(x-u-z-x+x_0)=A G_\alpha(x_0-u-z).$$
	
	Thus for such $x$ and $u$ it follows that 
	$$\varphi_1(x-u)\leq A \left[\int_{\RR^n} G_\alpha(x_0-u-z) d\mu(z)\right]^{\frac{1}{s-1}}= A \varphi(x_0-u). $$
	
	Hence,
	\begin{align*}
	V_{21}(x)&\leq A \int_{|x-u-x_0|\leq 3} G_\alpha(u)\varphi(x_0-u) du\\
	& \leq A \int_{\RR^n}G_\alpha(x_0-y)\varphi(y) dy\\
	&= A V(x_0).
	\end{align*}
	
	As this holds for all $|x-x_0|\leq 1$ we deduce that 
	\begin{equation}\label{E2}
	\fint_{B_r(x_0)} V^t_{21}(x) dx \leq A V^t(x_0).
	\end{equation}
	
	It is  now left to estimate $V_{22}$. Note that for $|x-x_0| \leq 1$,
	$$V_{22}(x)=\int_{|y-x_0|\leq 3}G_\alpha(x-y)\varphi_2(y) dy\leq \int_{|y-x|\leq 4}G_\alpha(x-y)\varphi_2(y) dy,$$
	and thus by \eqref{G1} we have 
	\begin{equation}\label{V221}
	V_{22}(x)\leq C \int_{0}^5 \frac{\int_{B_\rho(x)} \varphi_2(y) dy}{\rho^{n-\alpha}}\frac{d\rho}{\rho}.
	\end{equation}

	We next claim   that for $|x-x_0| \leq 1$,
	\begin{equation}\label{VJ}
	V_{22}(x)\leq C \int_{0}^{100} \left(\frac{\mu(B_\rho(x))}{\rho^{n-\alpha s}} \right)^{\frac{1}{s-1}}\frac{d\rho}{\rho}.
	\end{equation}
	
	Assuming \eqref{VJ}, we have 
	\begin{equation}\label{Q1Q2}
	\fint_{B_r(x_0)} V^t_{22}(x) dx \leq C (Q_1+Q_2),   
	\end{equation}
	where 
	$$Q_1=\fint_{B_r(x_0)} \left(\int_{0}^{r} \left(\frac{\mu(B_\rho(x))}{\rho^{n-\alpha s}} \right)^{\frac{1}{s-1}}\frac{d\rho}{\rho}\right)^t dx,$$
	$$Q_2= \fint_{B_r(x_0)}\left(\int_{r}^{100} \left(\frac{\mu(B_\rho(x))}{\rho^{n-\alpha s}} \right)^{\frac{1}{s-1}}\frac{d\rho}{\rho}\right)^t dx. $$
	
	For $Q_2$ we observe that if $x\in B_r(x_0)$ and $\rho\geq r$ then $B_\rho(x)\subset B_{2\rho}(x_0)$, and so 
	$$Q_2\leq  \fint_{B_r(x_0)}\left(\int_{r}^{100} \left(\frac{\mu(B_{2\rho}(x_0))}{\rho^{n-\alpha s}} \right)^{\frac{1}{s-1}}\frac{d\rho}{\rho}\right)^t dx\leq C V^t(x_0), $$
	where we used \cite[Theorem 2]{Ad1} in the last inequality.
	
	For $Q_1$, we first bound, with $x\in B_r(x_0)$ and $\epsilon>0$,
	\begin{align*}
	\left(\int_{0}^{r} \left(\frac{\mu(B_\rho(x))}{\rho^{n-\alpha s}} \right)^{\frac{1}{s-1}}\frac{d\rho}{\rho}\right)^t &\leq c\,  r^{\frac{\epsilon t}{s-1}} \sup_{0<\rho<r} 
	\left(\frac{\mu(B_\rho(x))}{\rho^{n-\alpha s +\epsilon}} \right)^{\frac{t}{s-1}} \\
	&\leq c\, r^{\frac{\epsilon t}{s-1}} \sup_{0<\rho<r}
	\left(\int_{|x-z|<\rho} \frac{d\mu(z)}{|x-z|^{n-\alpha s +\epsilon}} \right)^{\frac{t}{s-1}}\\
	&\leq c\, r^{\frac{\epsilon t}{s-1}}  \left(\int_{|x_0-z|<2r} \frac{d\mu(z)}{|x-z|^{n-\alpha s +\epsilon}} \right)^{\frac{t}{s-1}}.
	\end{align*}
	
	This and Minkowski's inequality (recall that $t>s-1$) yield 
	\begin{align*}
	Q_1 &\leq c\, r^{\frac{\epsilon t}{s-1}}  \fint_{B_r(x_0)} \left(\int_{|x_0-z|<2r} \frac{d\mu(z)}{|x-z|^{n-\alpha s +\epsilon}} \right)^{\frac{t}{s-1}} dx \\
	&\leq c\, r^{\frac{\epsilon t}{s-1}} \left[\int_{|x_0-z|<2r} d\mu(z) \left(\fint_{B_r(x_0)} \frac{dx}{|x-z|^{(n-\alpha s+\epsilon)\frac{t}{s-1}}}\right)^{\frac{s-1}{t}}  \right]^{\frac{t}{s-1}}.
	\end{align*}
	
	We now choose an $\epsilon>0$ such that  $(n-\alpha s+\epsilon)\frac{t}{s-1}<n$, which is possible since 
	$$\frac{n(s-1)}{t} > n-\alpha s.$$
	
	Then by  simple calculations and  \cite[Theorem 2]{Ad1} we arrive at 
	$$Q_1\leq c\, \left(\frac{\mu(B_{2r}(x_0))}{r^{n-\alpha s}}\right)^{\frac{t}{s-1}} \leq c\,
	\left(\int_{0}^{4} \left(\frac{\mu(B_{\rho}(x_0))}{\rho^{n-\alpha s}} \right)^{\frac{1}{s-1}}\frac{d\rho}{\rho}\right)^t dx\leq C V^t(x_0).$$

	Thus in view of \eqref{Q1Q2} and the above estimates for $Q_1$ and $Q_2$ we get 
	\begin{equation}\label{E3}
	\fint_{B_r(x_0)} V^t_{22}(x) dx \leq A V^t(x_0). 
	\end{equation}
	
	Estimates \eqref{E1}, \eqref{E2}, and \eqref{E3} yield the bound \eqref{AV} as desired.

	Therefore, what's left now is to verify inequality \eqref{VJ}. In view of \eqref{V221} we need to estimate 
	$\int_{B_\rho(x)} \varphi_2(y) dy$.
	To this end, we first observe that for $y\in B_\rho(x)$, $\rho\in(0,5]$, by \eqref{G1} it holds that 
	\begin{align*}
	\varphi_2(y)&=\left(\int_{|z-x_0|\leq 5} G_\alpha(y-z) d \mu(z)\right)^{\frac{1}{s-1}}\\
	&\leq \left(\int_{|z-y|\leq 11} G_\alpha(y-z) d \mu(z)\right)^{\frac{1}{s-1}}\\
	&\leq C \left(\int_{|z-y|\leq 11} \frac{1}{|y-z|^{n-\alpha}} d \mu(z)\right)^{\frac{1}{s-1}}.
	\end{align*}	
	
	Thus for $0<\rho\leq 5$ we have 
	\begin{equation}\label{V222}
	\int_{B_\rho(x)} \varphi_2(y) dy \leq C (I_1 +I_2), 
	\end{equation}
	
	where 
	$$I_1=\int_{B_\rho(x)} \left(\int_{|z-y| <\rho} \frac{1}{|y-z|^{n-\alpha}} d \mu(z)\right)^{\frac{1}{s-1}} dy,$$
	and
	\begin{align}\label{I2}
	I_2 &=\int_{B_\rho(x)} \left(\int_{\rho\leq |z-y|\leq 11} \frac{1}{|y-z|^{n-\alpha}} d \mu(z)\right)^{\frac{1}{s-1}} dy \\
	& \leq \int_{B_\rho(x)} \left( \int_{\rho}^{12} \frac{\mu(B_t(y))}{t^{n-\alpha}} \frac{dt}{t} \right)^{\frac{1}{s-1}} dy\nonumber\\
	& \leq c\, \rho^n  \left( \int_{\rho}^{12} \frac{\mu(B_{2t}(x))}{t^{n-\alpha}} \frac{dt}{t} \right)^{\frac{1}{s-1}}.\nonumber
	\end{align}
	
	We now claim that 
	\begin{equation}\label{I1}
	I_1\leq c\, \mu(B_{2\rho}(x))^\frac{1}{s-1} \rho^{n+\frac{\alpha-n}{s-1}}.
	\end{equation}
	
	Indeed, we have 
	$$I_1\leq \int_{|x-y|\leq \rho} \left(  \int_{|z-x|<2\rho} \frac{d\mu(z)}{|z-y|^{n-\alpha}} \right)^{\frac{1}{s-1}} dy, $$
	and thus when $s>2$ by H\"older's inequality with exponents $s-1$ and $\frac{s-1}{s-2}$ and Fubini's theorem we obtain 
	\begin{align*}
	I_1 &\leq \left(\int_{|z-x|<2\rho} d\mu(z) \int_{|x-y|\leq \rho} \frac{dy}{|z-y|^{n-\alpha}} \right)^{\frac{1}{s-1}} |B_\rho(x)|^{\frac{s-2}{s-1}}\\
	&\leq c\, \mu(B_{2\rho}(x))^\frac{1}{s-1}  \rho^{\frac{\alpha}{s-1}} \rho^{\frac{n(s-2)}{s-1}}=c\, \mu(B_{2\rho}(x))^\frac{1}{s-1} \rho^{n+\frac{\alpha-n}{s-1}}.
	\end{align*}
	
	On the other hand, when $2-\alpha/n<s\leq 2$ we use Minkowski's inequality to get 
	\begin{align*}
	I_1 &\leq \left[\int_{|z-x|<2\rho}  \left(\int_{|x-y|\leq \rho} \frac{dy}{|z-y|^{\frac{n-\alpha}{s-1}}}\right)^{s-1} d\mu(z) \right]^{\frac{1}{s-1}} \\
	&\leq c\, \mu(B_{2\rho}(x))^\frac{1}{s-1} \rho^{n+\frac{\alpha-n}{s-1}}.
	\end{align*} 
	
	Thus the claim \eqref{I1} follows.
	At this point combining estimates \eqref{V221}, \eqref{V222}-\eqref{I1} we obtain 
	\begin{equation}\label{VJ1}
	V_{22}(x)\leq C \int_{0}^{5} \left(\frac{\mu(B_\rho(x))}{\rho^{n-\alpha s}} \right)^{\frac{1}{s-1}} \frac{d\rho}{\rho} + C  J(x),
	\end{equation}
	where 
	$$J(x)=\int_{0}^{5}  \rho^\alpha  \left(\int_{\rho}^{12} \frac{\mu(B_{2t}(x))}{t^{n-\alpha}} \frac{dt}{t} \right)^{\frac{1}{s-1}}  \frac{d\rho}{\rho}.$$
	
	Using Hardy's inequality of the form 
	$$\left\{\int_{0}^\infty  \left(\int_{\rho}^\infty f(t) dt \right)^{q}  \rho^{\alpha} \frac{d\rho}{\rho} \right\}^{\frac{1}{q}} \leq \frac{q}{\alpha}
	\left\{\int_{0}^\infty (t f(t))^q t^\alpha\frac{dt}{t} \right\}^{\frac{1}{q}}, $$
	$1\leq q<\infty$, $\alpha>0$, $f\geq 0$, when $s\leq 2$ we find
	$$J(x)\leq C \int_{0}^{12} \left(\frac{\mu(B_{2t}(x))}{t^{n-\alpha s}} \right)^{\frac{1}{s-1}}\frac{dt}{t}\leq C \int_{0}^{24} \left(\frac{\mu(B_{\rho}(x))}{\rho^{n-\alpha s}} \right)^{\frac{1}{s-1}}\frac{d\rho}{\rho}.$$

	When $s>2$ we have
	\begin{align*}
	\left(\int_{\rho}^{12} \frac{\mu(B_{2t}(x))}{t^{n-\alpha}} \frac{dt}{t} \right)^{\frac{1}{s-1}} &\leq 
	\left(\sum_{ k=0}^{k_0} \int_{ 2^k\rho}^{2^{k+1}\rho} \frac{\mu(B_{2t}(x))}{t^{n-\alpha}} \frac{dt}{t} \right)^{\frac{1}{s-1}}\\
	& \leq C\, \left(\sum_{ k=0}^{k_0}  \frac{\mu(B_{2^{k+2}\rho}(x))}{(2^k\rho)^{n-\alpha}}  \right)^{\frac{1}{s-1}}\\
	& \leq C\, \sum_{ k=0}^{k_0}  \left(\frac{\mu(B_{2^{k+2}\rho}(x))}{(2^k\rho)^{n-\alpha}}\right)^{\frac{1}{s-1}},
	\end{align*}
	where  $k_0=k_0(\rho)$ is an integer such that $2^{k_0+1}\rho \geq 12$ and $2^{k_0+1}\rho <  24$. 
	This yields
	$$\left(\int_{\rho}^{12} \frac{\mu(B_{2t}(x))}{t^{n-\alpha}} \frac{dt}{t} \right)^{\frac{1}{s-1}} \leq  C\, \int_{2\rho}^{48} \left(\frac{\mu(B_{2t}(x))}{t^{n-\alpha}}\right)^{\frac{1}{s-1}} \frac{dt}{t} .$$

	Thus by Fubini's theorem we get 
	\begin{align*}
	J(x) &\leq  C\, \int_{0}^5 \rho^\alpha  \int_{2\rho}^{48} \left(\frac{\mu(B_{2t}(x))}{t^{n-\alpha}}\right)^{\frac{1}{s-1}} \frac{dt}{t}  \frac{d\rho}{\rho}\\
	& \leq  C\,\int_{0}^{48} \left(\frac{\mu(B_{2t}(x))}{t^{n-\alpha}}\right)^{\frac{1}{s-1}}  \int_{0}^{t/2}\rho^\alpha \frac{d\rho}{\rho}   \frac{dt}{t}\\
	&= C \int_{0}^{48}   \left(\frac{\mu(B_{2t}(x))}{t^{n-\alpha s}}\right)^{\frac{1}{s-1}} \frac{dt}{t}\\
	&\leq C \int_{0}^{96}   \left(\frac{\mu(B_{\rho}(x))}{\rho^{n-\alpha s}}\right)^{\frac{1}{s-1}} \frac{dt}{t}.
	\end{align*}
	
	Now combining \eqref{VJ1} with the above estimates for $J(x)$ we arrive at \eqref{VJ} as desired.
	
	The proof of the theorem is complete. 
\end{proof}

For any set $E\subset{\RR^n}$ with $0<\text{Cap}_{\alpha,s}(E)<\infty$, by \cite[Theorems 2.5.6 and 2.6.3 ]{AH}
one can find a nonnegative measure $\mu=\mu^{E}$ with ${\rm supp}(\mu) \subset \overline{E}$ (called capacitary measure for $E$) such that the function 
 $V^{E}=G_\alpha*((G_\alpha*\mu)^{\frac{1}{s-1}})$ satisfies the following properties:

\begin{equation}\label{muEcap}
\mu^E(\overline{E})={\rm Cap}_{\alpha,s}(E)=\int_{\RR^n} V^E d\mu^E= \int_{\RR^n} (G_\alpha *\mu^{E})^{\frac{s}{s-1}} dx,
\end{equation}
\begin{equation*}
V^E\geq 1 \quad \text{quasieverywhere on }E,
\end{equation*}
and
\begin{equation}\label{VElA}
V^E\leq A \quad \text{on } \RR^n.
\end{equation}

\begin{lemma}\label{L1CforVEt} Let $E$, $\mu=\mu^E$, and $V^E$ be as above and let $0<\alpha s\leq n$. If  $\delta\in (1, n/(n-\alpha))$ for $s< 2$ and $\delta\in (s-1, n(s-1)/(n-\alpha s))$ 
	for $s\geq 2$, then the function $(V^E)^\delta \in A_1^{\rm loc}$ with $[(V^E)^\delta]_{ A_1^{\rm loc}}\leq c(n, \alpha, s, \delta)$. Moreover, $(V^E)^\delta\in L^1(C)$ with 
	$\|(V^E)^\delta\|_{L^1(C)} \leq C\, {\rm Cap}_{\alpha, s}(E)$. 	
\end{lemma}

\begin{proof}
Thanks to Theorem 	\ref{mainA1loc}, we just meed to prove the last statement of the lemma. By \cite[Proposition 6.1.2]{AH} we see that $V^E$ and hence $(V^E)^\delta$ are quasicontinuous.
We have 
\begin{equation}\label{L1forVt}
\|(V^E)^\delta\|_{L^1(C)}= \delta \int_0^\infty {\rm Cap}_{\alpha, s}(\{V^E >\rho\}) \rho^{\delta-1}d\rho.
\end{equation}

For $s\geq2$, by \cite[Proposition 4.4]{AM} and \eqref{muEcap} it holds that 
$${\rm Cap}_{\alpha, s}(\{V^E >\rho\})\leq C\, \mu^{E}(\RR^n) \rho^{1-s} =C\, {\rm Cap}_{\alpha,s}(E) \rho^{1-s}. $$

For $1<s<2$, let $\nu$ be the capacitary measure for the set  $\{V^E >\rho\}$. By Fubini's theorem we have 
\begin{align*}
{\rm Cap}_{\alpha, s}(\{V^E >\rho\})&=\int_{\RR^n} d\nu \leq \rho^{-1} \int_{\RR^n} V^E d\nu\\
&=\rho^{-1} \int_{\RR^n} (G_\alpha *\mu^E(y))^{\frac{1}{s-1}} (G_\alpha*\nu(y)) dy.
\end{align*}

Thus by H\"older's inequality it follows that 
\begin{align*}
{\rm Cap}_{\alpha, s}(\{V^E >\rho\}) & \leq \rho^{-1}\left\{\int_{\RR^n} (G_\alpha *\mu^E(y))^{\frac{s}{s-1}} dy\right\}^{2-s} \times\\
&\quad \times  \left\{\int_{\RR^n} (G_\alpha *\mu^E(y)) (G_\alpha*\nu(y))^{\frac{1}{s-1}} dy\right\}^{s-1}\\
& = \rho^{-1} {\rm Cap}_{\alpha, s}(E)^{2-s} \int_{\RR^n} G_\alpha*((G_\alpha*\nu)^{\frac{1}{s-1}}) d\mu^{E}\\
&\leq C\, \rho^{-1} {\rm Cap}_{\alpha, s}(E)^{2-s} \mu^{E}(\RR^n)^{s-1}= C\, \rho^{-1} {\rm Cap}_{\alpha, s}(E).
\end{align*}

Using \eqref{VElA}-\eqref{L1forVt} and the above estimates for ${\rm Cap}_{\alpha,s}(\{V^E >\rho\})$,    we get 
\begin{align*}
\|(V^E)^\delta\|_{L^1(C)}\leq C\, {\rm Cap}_{\alpha, s}(E) \int_0^{A} \rho^{-\max\{s-1,1\}} \rho^{\delta-1} d\rho \leq C\, {\rm Cap}_{\alpha, s}(E),
\end{align*}
as desired.
\end{proof}	

We are now ready to prove Theorem \ref{firsttheorem}.

\begin{proof}[Proof of Theorem \ref{firsttheorem}]
	 By Theorem \ref{first dual}, given $w\in L^{1}(C)$ with $\|w\|_{L^{1}(C)}\leq 1$,  one has 
	\begin{align*}
	\int|f(x)|^{p} w(x)dx&\leq\| |f|^p\|_{M_{1}^{\alpha,s}}\int_{0}^{\infty}\text{Cap}_{\alpha,s}(\{w>t\})dt\\
	&=\|f\|^{p}_{M_{p}^{\alpha,s}}\| w\|_{L^{1}(C)},
	\end{align*}
	which yields
	\begin{align*}
	\sup\left\{\int|f(x)|^{p} w(x)dx: w\in L^{1}(C),\|w\|_{L^{1}(C)}\leq 1\right\}
	\leq\|f\|^{p}_{M_{p}^{\alpha,s}}.
	\end{align*}
	
	On the other hand, fix a constant $\delta$ such that $\delta\in (1, n/(n-\alpha))$ if $s< 2$ and $\delta\in (s-1, n(s-1)/(n-\alpha s))$ 
	if $s\geq 2$, and let $E$ be a compact subset of $\RR^n$ with ${\rm Cap}_{\alpha,s}(E)>0$. Then, with $V^E$ as in  Lemma \ref{L1CforVEt}, we can find a constant 
	${\bf \overline{c}}={\bf \overline{c}}(n,\alpha,s)\geq 1$ such that $[V^E/{\rm Cap}_{\alpha,s}(E)]_{A^{\rm loc}_{1}}\leq {\bf \overline{c}}$. Thus
	we have 
	\begin{align*}
	& \dfrac{\displaystyle\int_{E}|f(x)|^{p}dx}{\text{Cap}_{\alpha,s}(E)}\leq\dfrac{\displaystyle\int_{E}|f(x)|^{p}  (V^E)^{\delta}dx}{\text{Cap}_{\alpha,s}(E)}\leq \int_{\RR^n}|f(x)|^{p}\left(\dfrac{  (V^E)^{\delta}}{\text{Cap}_{\alpha,s}(E)}\right)dx\\
	& \leq C \, \sup\left\{\int|f(x)|^{p} w(x)dx: w\in L^{1}(C),\|w\|_{L^{1}(C)}\leq 1, [w]_{A^{\rm loc}_{1}}\leq {\bf \overline{c}}\right\}.
	\end{align*}
	
	This finishes the proof of the theorem.
\end{proof}

\begin{proof}[Proof of Theorem \ref{second}]
If $f\in M^{\alpha,s}_{p}$, then for any $g\in N^{\alpha,s}_{p'}$ and $w\in L^1(C)$, $\|w\|_{L^1(C)}\leq 1$, such that $\int |g|^{p'} w^{1-p'} dx<+\infty$,  one has
\begin{align*}
\left|\int f(x) g(x) dx\right| \leq \left(\int |f|^p w dx\right)^{\frac{1}{p}} \left(\int |g|^{p'} w^{1-p'} dx\right)^{\frac{1}{p'}}   
\end{align*} 
by H\"older's inequality. Thus  it follows from the proof of Theorem \ref{firsttheorem} that 
$$\left|\int f(x) g(x) dx\right| \leq  \| f\|_{M^{\alpha,s}_{p}} \| g\|_{N^{\alpha,s}_{p'}},$$
and so $L_{f}\in (N^{\alpha,s}_{p'})^*$.

Conversely, let $L \in (N^{\alpha,s}_{p'})^*$ be given. If $g\in L^{p'}$ with ${\rm supp}(g)\subset E$ for a bounded set $E$ with positive capacity, then with $V^E$ and $\delta$ as in Lemma 
\ref{L1CforVEt} we have  
$$\int_{\RR^n} |g|^{p'} \left[(V^E)^\delta/{\rm Cap}_{\alpha,s}(E)\right]^{1-p'} dx\leq {\rm Cap}_{\alpha,s}(E)^{p'-1} \|g\|_{L^{p'}}^{p'}.$$

Thus $g\in N^{\alpha,s}_{p'}$ with 
$$\|g\|_{N^{\alpha,s}_{p'}} \leq C {\rm Cap}_{\alpha,s}(E)^{\frac{1}{p}} \|g\|_{L^{p'}},$$
and so
$$|L(g)| \leq C \|L \|  {\rm Cap}_{\alpha,s}(E)^{\frac{1}{p}} \|g\|_{L^{p'}}.$$

By Riesz's representation theorem there is an $f\in L^p_{\rm loc}(\RR^n)$ such that 
\begin{equation}\label{Rpre}
L(g)=\int f(x) g(x) dx  
\end{equation} 
for all $g\in L^{p'}$ with compact support. In particular, if $g={\rm sgn}(f)|f|^{p-1} \chi_{K}$ for any compact set $K$, then we deduce
\begin{align*}
|L(g)|&= \int_{K} |f|^p \leq C \|L \|  {\rm Cap}_{\alpha,s}(K)^{\frac{1}{p}} \|g\|_{L^{p'}}\\
&=C \|L \|  {\rm Cap}_{\alpha,s}(K)^{\frac{1}{p}} \left(\int_{K} |f|^p\right)^{\frac{1}{p'}}.
\end{align*}

This implies $f\in M^{\alpha, s}_p$ and 
$$\|f\|_{M^{\alpha, s}_p} \leq C \|L \|.$$

Note that for any $g\in N^{\alpha, s}_{p'}$, the functions 
$$g_{k}:=\max\{\min\{g,k\}, -k\}\chi_{B_k(0)}, \quad k\geq 1,$$
converge to $g$ in $N^{\alpha, s}_{p'}$ as $k\rightarrow\infty$. Also, for any $g\in N^{\alpha, s}_{p'}$  and $k\geq 1$ we have 
$$\int |f| |g|_{k} dx \leq C \| f\|_{M^{\alpha,s }_p} \| g_k\|_{N^{\alpha,s }_{p'}} \leq C \| f\|_{M^{\alpha,s }_p} \| g\|_{N^{\alpha,s }_{p'}},$$
and thus by Fatou's lemma we get $f g\in L^{1}(\RR^n)$. Then by continuity, \eqref{Rpre}, and Lebesgue Dominated Convergence Theorem we arrive at
$$L(g)=\lim_{k\rightarrow\infty} L(g_k)=\lim_{k\rightarrow\infty} \int f(x) g_k(x) dx= \int f(x) g(x) dx$$
for all $g\in N^{\alpha, s}_{p'}$.

This completes the proof of the theorem.
\end{proof}

\begin{remark}
	The above proof shows that bounded functions with compact support $f$ are dense in $N^{\alpha, s}_{q}$. For such  $f$, we define $\rho_\epsilon * f= \epsilon^{-n}\rho(\epsilon^{-1} \cdot)*f$, where $\epsilon\in(0,1)$ and $\rho\in C_c^\infty(B_1(0))$. Let $B$ be a ball such that 
	${\rm supp}(f)\subset B$ and ${\rm supp}(\rho_\epsilon * f)\subset B$ for any $\epsilon\in(0,1)$. Then take a weight $w\in L^1(C)\cap A^{\rm loc}_1$ such that 
	$w\geq 1$ on $B$. We have
	$$\norm{\rho_\epsilon * f -f}_{N^{\alpha, s}_{q}}\leq C \left(\int_{\RR^n}|\rho_\epsilon * f -f|^q w^{1-q} dx\right)^{\frac{1}{q}}\leq C \norm{\rho_\epsilon * f -f}_{L^q}.$$

	Thus we see that  $C_c^\infty(\RR^n)$ is dense in 	$N^{\alpha, s}_{q}$. 
	Likewise, we also have  that $C_c^\infty(\RR^n)$ is dense in  the space	$\mathcal{N}^{\alpha, s}_{q}$.
\end{remark}

\begin{proof}[Proof of Theorem \ref{NoA1Th}] 
One just needs to follow the proofs of Theorems \ref{firsttheorem} and \ref{second} and replace the function $(V^E)^{\delta}$ with the characteristic function $\chi_E$.	
\end{proof}

\begin{remark} In general, functions in $N^{\alpha, s}_{p'}$ (hence $\widetilde{N}^{\alpha, s}_{p'}$) do not belong to $L^{p'}_{\rm loc}(\RR^n)$.
	To see this, consider the case $p=2, \alpha=1/4, s=2,$ and $n\geq 3$. Let $g(x)=|x|^{-n+1}$ for $|x|<1$ and  for $g(x)=|x|^{-n-1}$ for $|x|\geq 1$. 
	Also, let $w(x)=g(x)$ for any $x\in  \RR^n$.  Then using \eqref{rleq1} and \eqref{rgeq1}	it can be shown that $w\in A_1^{\rm loc}\cap L^1(C)$. Moreover, we have $g^{2} w^{-1} \in L^1(\RR^n)$. 
	Thus $g\in N^{\alpha, s}_{p'}$ (by enlarging ${\bf \overline{c}}(n,\alpha,s)$ if necessary) but $g\not\in L^{2}(B_1(0))$.
	
	On the other hand, if in the definition of  $N^{\alpha, s}_{p'}$ 
	we consider only weights $w$ such that $w \in L^\infty(\RR^n) \cap L^{1}(C)\cap A^{\rm loc}_1$ with $\|w\|_{L^{1}(C)}\leq 1$ and $[w]_{A^{\rm loc}_1}\leq {\bf \overline{c}}(n,\alpha,s)$
	(or only weights $w$ such that $w \in L^\infty(\RR^n)$ and $\int_{\RR^n} w dC \leq 1$ for $\widetilde{N}^{\alpha, s}_{p'}$), then Theorems \ref{second} and \ref{NoA1Th} still remain valid for those versions of $N^{\alpha, s}_{p'}$ and  $\widetilde{N}^{\alpha, s}_{p'}$. Moreover, functions in such spaces belong to $L^{p'}(\RR^n)$.
\end{remark}

\section{Proof of Theorem \ref{BspaceTh}}

\begin{proof}[Proof of Theorem \ref{BspaceTh}] By Proposition \ref{M'*}, we just need to show $(B^{\alpha,s}_{p'})^*=M^{\alpha,s}_{p}$.
Let $f\in M^{\alpha,s}_{p}$ and $g\in B^{\alpha,s}_{p'}$. Suppose that 
$g=\sum_{j} c_j a_j$ where $a_j=0$ in $\RR^n\setminus A_j$ and $\|a_j\|_{L^{p'}}\leq {\rm Cap}_{\alpha,s}(A_j)^{-1/p}$. We have 
\begin{align*}
 \left|\int f(x) g(x) dx\right| &\leq \sum_{j} |c_j| \int_{A_j} |f a_j| dx \\
&\leq \sum_{j} |c_j| \|f\|_{L^p(A_j)} \|a_j\|_{L^{p'}}\\
& \leq \sum_{j} |c_j| \|f\|_{L^p(A_j)} {\rm Cap}_{\alpha,s}(A_j)^{-1/p}\\
& \leq \sum_{j} |c_j|  \|f\|_{M^{\alpha,s}_p}.
\end{align*}

Thus, 
\begin{equation}\label{Fpart}
\left|\int f(x) g(x) dx\right| \leq  \|f\|_{M^{\alpha,s}_p}  \|g\|_{B^{\alpha,s}_{p'}}. 
\end{equation}

Conversely, let $L\in B^{\alpha, s}_{p'}$ be given. If $0\not=g\in L^{p'}$ with ${\rm supp} (g) \subset E$ for a bounded set $E$ then  $g\in B^{\alpha,s}_{p'}$ as 
we can write $g={\rm Cap}_{\alpha,s}(E)^{1/p} \|g\|_{L^{p'}} \tilde{g},$ where $\tilde{g}=g/({\rm Cap}_{\alpha,s}(E)^{1/p} \|g\|_{L^{p'}})$, and so 
$$\|g\|_{B^{\alpha,s}_{p'}} \leq  {\rm Cap}_{\alpha,s}(E)^{\frac{1}{p}} \|g\|_{L^{p'}}.$$

This gives
$$|L(g)| \leq  \|L \|  {\rm Cap}_{\alpha,s}(E)^{\frac{1}{p}} \|g\|_{L^{p'}}.$$

Then as in the proof of Theorem \ref{second} we can find an $f\in M^{\alpha, s}_p$ with $\|f\|_{M^{\alpha, s}_p} \leq  \|L \|,$
and 
\begin{equation}\label{BprimeRpre}
L(g)=\int f(x) g(x) dx  
\end{equation} 
for all $g\in L^{p'}$ with compact support.

We will now show that \eqref{BprimeRpre} holds for all $g\in B^{\alpha, s}_{p'}$. Note that for any $g\in B^{\alpha, s}_{p'}$, 
we have a representation 
$g=\sum_{j} c_j a_j$ where $a_j=0$ in $\RR^n\setminus A_j$, $A_j$'s are bounded sets, $\|a_j\|_{L^{p'}}\leq {\rm Cap}_{\alpha,s}(A_j)^{-1/p}$, and $\sum_{j} |c_j| <+\infty.$
Thus the functions 
$$g_{k}:=\sum_{|j|\leq k} c_j a_j, \qquad k\geq 1,$$
have compact support and converge to $g$ in $B^{\alpha, s}_{p'}$ as $k\rightarrow\infty$.
Also, if $h_k=\sum_{|j|\leq k} |c_j| |a_j|$, $k\geq 1$, then $h_k\in B^{\alpha, s}_{p'}$ and $\|h_k\|_{B^{\alpha,s}_{p'}} \leq \sum_{j} |c_j|$. 
Thus using \eqref{Fpart} we have 
\begin{equation*}
\int |f| h_{k} dx \leq  \| f\|_{M^{\alpha,s }_p} \| h_k\|_{B^{\alpha,s }_{p'}} \leq  \| f\|_{M^{\alpha,s }_p} \sum_{j} |c_j|,
\end{equation*}
and  by Fatou's lemma we get $f\sum_{j} |c_j| |a_j|$ and  $f g\in L^{1}(\RR^n)$.

Now by continuity, \eqref{BprimeRpre}, and Lebesgue Dominated Convergence Theorem we arrive at
$$L(g)=\lim_{k\rightarrow\infty} L(g_k)=\lim_{k\rightarrow\infty} \int f(x) g_k(x) dx= \int f(x) g(x) dx$$
for all $g\in B^{\alpha, s}_{p'}$.
\end{proof}	

\begin{remark} The proof above shows that if $g=\sum_{j} c_j a_j$, where $a_j=0$ in $\RR^n\setminus A_j$ for  a bounded set $A_j$, $\|a_j\|_{L^{p'}}\leq {\rm Cap}_{\alpha,s}(A_j)^{-1/p}$, and $\sum_{j} |c_j| <+\infty$, then the series 
$\sum_{j} c_j a_j$   converges absolutely a.e. in $\RR^n$. 
\end{remark}

\section{The space $\mathcal{N}_{q}^{\alpha,s}$}

Let $\alpha>0, s>1$, and $q>1$. Recall that   $\mathcal{L}^{1}(C)$ is defined as the K\"othe dual of $M^{\alpha,s}_{1}$.  Also, by Theorem \ref{first dual} we see that $L^1(C)$ is continuously embedded into $\mathcal{L}^1(C)$. Moreover, the norm of a function $g\in \mathcal{N}^{\alpha,s}_q$ is the defined as
\begin{align}\label{Nmathcalnorm}
\|g\|_{\mathcal{N}^{\alpha,s}_q}=\inf_{w}\left(\int_{\mathbb{R}^n}|g(x)|^{q}w(x)^{1-q}dx\right)^{1/q},
\end{align}
where the infimum is taken over all $w\in \mathcal{L}^{1}(C)\cap A^{\rm loc}_1$ with $\|w\|_{\mathcal{L}^{1}(C)}\leq 1$ and $[w]_{A^{\rm loc}_1}\leq {\bf \overline{c}}(n,\alpha,s)$. 

We remark that  if ${\rm Cap}_{\alpha,s}$ is strongly subadditive, then it can be shown from Hahn-Banach Theorem, Theorem  \ref{first dual}, and an approximation argument that  $\norm{f}_{L^1(C)} = \norm{f}_{\mathcal{L}^1(C)}$ for all $f\in L^1(C)$. 

Let  $\delta$ be a fixed constant such that $\delta\in (1, n/(n-\alpha))$ if $s< 2$ and $\delta\in (s-1, n(s-1)/(n-\alpha s))$ 
if $s\geq 2$.
We observe that if $E$ is subset of $\RR^n$  such that  $0<\text{Cap}_{\alpha,s}(E)<\infty$ and $V^E$  is as in Lemma \ref{L1CforVEt}, then
by Theorem \ref{first dual} and Lemma \ref{L1CforVEt} we have 
$$\int |g(x)| \left(\dfrac{  (V^E)^{\delta}}{\text{Cap}_{\alpha,s}(E)}\right)dx \leq C $$
for all $g\in L^1_{\rm loc} \cap M_1^{\alpha,s}$ such that $\|g\|_{M_1^{\alpha,s}}\leq 1$. That is, $$\left\|(V^E)^{\delta}/\text{Cap}_{\alpha,s}(E)\right\|_{\mathcal{L}^1(C)}\leq C.$$

 Moreover, $[V^E/{\rm Cap}_{\alpha,s}(E)]_{A^{\rm loc}_{1}}\leq {\bf \overline{c}}(n,\alpha,s)$ for some ${\bf \overline{c}}(n,\alpha,s)\geq 1$.
  Thus by a simple modification of the proofs of Theorems \ref{firsttheorem}
and \ref{second} we obtain the following duality result.

\begin{theorem}\label{NmathcaloA1Th} 
	For $p>1$ and $\alpha>0, s>1$, with $\alpha s\leq n$, we have 
	\begin{align*}
	\|f\|_{M^{\alpha,s}_p} \simeq \sup_{w}\left(\int_{{\RR}^{n}}|f(x)|^{p} w(x)dx\right)^{1/p},
	\end{align*}
	where the supremum is taken over all weights  $w\in \mathcal{L}^{1}(C)\cap A^{\rm loc}_1$ with $\|w\|_{\mathcal{L}^{1}(C)}\leq 1$ and $[w]_{A^{\rm loc}_1}\leq {\bf \overline{c}}(n,\alpha,s)$.  	Moreover, we have  $\left(\mathcal{N}^{\alpha,s}_{p'}\right)^{\ast} \approx M^{\alpha,s}_p$.
\end{theorem}

The fact that $\|\cdot \|_{\mathcal{L}^1(C)}$ is a norm yields the following important result. A related result in the setting of  Morrey spaces can be found in
 \cite{MST}.

\begin{theorem}\label{NisBFS} The space $\mathcal{N}^{\alpha,s}_q$ with the norm given by \eqref{Nmathcalnorm} is a Banach function space.
\end{theorem}
\begin{proof}
 For any $g\in \mathcal{N}^{\alpha,s}_q$ we set 
\begin{equation*}
\|g\|_1:= \inf\Big\{\sum_{j}|c_j|: g= \sum_{j} c_j b_j \text{ a.e.}\Big\},
\end{equation*}
where each $b_j$ is a block in  $\mathcal{N}^{\alpha,s}_q$, i.e., $b_j\in \mathcal{N}^{\alpha,s}_q$ and $\|b_j\|_{\mathcal{N}^{\alpha,s}_q}\leq 1$. It is easy to see that
$\|\cdot\|_1$ is actually a norm and $(\mathcal{N}^{\alpha,s}_q, \|\cdot\|_1)$ is a Banach space. That $\|g\|_1=0$ implies $g=0$ a.e. can be checked as follows. Since $\|g\|_1=0$, for any $\epsilon>0$, there exist $\{c_j\}\in \ell^1$ and blocks $b_j$'s such that 
$g= \sum_j c_j b_j$ and $\sum_j |c_j| <\epsilon$. Then for any $\varphi\in C_c(\RR^n)$ by Theorem \ref{NmathcaloA1Th} we have 
$$\int |g| |\varphi| dx \leq \sum_j |c_j| \int |b_j| |\varphi| dx\leq \sum_j |c_j| \| b_j\|_{\mathcal{N}^{\alpha,s}_q} \|\varphi\|_{M^{\alpha,s}_{q'}}\leq \epsilon\,  \|\varphi\|_{M^{\alpha,s}_{q'}},$$
which yields that $\int |g| |\varphi| dx=0$ for all $\varphi\in C_c(\RR^n)$ and hence $g=0$ a.e.

We next  show that 
\begin{equation}\label{norm1}
\|g\|_{\mathcal{N}^{\alpha,s}_q} =\|g\|_{1}
\end{equation} 
for all $g\in \mathcal{N}^{\alpha,s}_q$ and thus property (P1) in the definition of Banach function space is fulfilled (see Sub-section \ref{BFS}). 
 That $\|g\|_{1}\leq \|g\|_{\mathcal{N}^{\alpha,s}_q}$ is obvious. To show the converse, we will show that 
\begin{equation}\label{1plusep}
\|g\|_{\mathcal{N}^{\alpha,s}_q}\leq (1+\epsilon)^2 \|g\|_{1}, \qquad \forall \epsilon>0.
\end{equation} 

For any $g\in \mathcal{N}^{\alpha,s}_q$, $g\not=0$, and any $\epsilon>0$, there exist $\{c_j\}\in \ell^1$ and blocks $b_j$'s such that 
$g= \sum_j c_j b_j$ and $$\sum_j |c_j| \leq (1+\epsilon) \|g\|_{1}.$$

Since $\|b_j\|_{\mathcal{N}^{\alpha,s}_q}\leq 1$, we can find $w_j\in \mathcal{L}^1(C)\cap A_1^{\rm loc}$, with $\|w_j\|_{\mathcal{L}^1(C)}\leq 1$ and 
$[w_j]_{A_1^{\rm loc}}\leq {\bf \overline{c}}(n,\alpha,s)$, such that 
$$\left(\int |b_j|^{q} w_j^{1-q} dx\right)^{\frac{1}{q}} \leq 1+\epsilon.$$

By H\"older's inequality we have 
\begin{align*}
|g|^q \leq \Big(\sum_j |c_j||b_j|\Big)^{q} \leq \Big(\sum_j |c_j| w_j\Big)^{q-1} \Big(\sum_j c_j |b_j|^q w_j^{1-q} \Big).
\end{align*}

Let $w=\sum_{j}|c_j| w_j$. It is easy to see that $w\in \mathcal{L}^1(C)\cap A_1^{\rm loc}$, with $\|w\|_{\mathcal{L}^1(C)}\leq \sum_j |c_j|$ and 
$[w]_{A_1^{\rm loc}}\leq {\bf \overline{c}}(n,\alpha,s)$. We then have

$$\int |g|^q w^{1-q} dx \leq \sum_j |c_j|\int |b_j|^q |w_j|^{1-q}dx \leq \sum_j |c_j| (1+\epsilon)^q.$$

This gives
$$\int |g|^q \Big(\dfrac{w}{\sum_j |c_j|}\Big)^{1-q} dx \leq \Big(\sum_j |c_j|\Big)^q (1+\epsilon)^q,$$
and so 
$$\|g\|_{\mathcal{N}^{\alpha,s}_{q}} \leq \sum_j |c_j| \,  (1+\epsilon)\leq (1+\epsilon)^2 \|g\|_{1}.$$

Thus we obtain \eqref{1plusep} and so \eqref{norm1} follows.

As properties (P3) and (P4) are easy to check, what's left now is to verify the Fatou property (P2). To this end, let $\{f_j\}$, $j=1,2, \dots$, be a sequence of nonnegative measurable functions in $\mathcal{N}^{\alpha,s}_{q}$ and $f_j \uparrow f$ a.e. in $\RR^n$.  We just need to show that $\|f\|_{\mathcal{N}^{\alpha,s}_{q}} \leq \sup_{j\geq 1} \|f_{j}\|_{\mathcal{N}^{\alpha,s}_{q}}$. For this, we may assume that 
  $$\sup_{j\geq 1}\|f_{j}\|_{\mathcal{N}^{\alpha,s}_q}= M<+\infty.$$

Then for any $j\geq 1$ and $\epsilon>0$ we can find  $w_{j}\in \mathcal{L}^{1}(C)\cap A_1^{\rm loc}$, with $\|w_{j}\|_{\mathcal{L}^1(C)}\leq 1$ and 
$[w_{j}]_{A_1^{\rm loc}}\leq {\bf \overline{c}}(n,\alpha,s)$, 
such that  
\begin{align}\label{Mplus}
\left(\int|f_{j}(x)|^{q} w_{j}(x)^{1-q}dx\right)^{\frac{1}{q}} \leq M+\epsilon.
\end{align}

Note that if $g=1/C$, where $C$ is the constant in \eqref{LinftyM} then we have 
$g\in M_1^{\alpha,s}$ and $\|g\|_{M_1^{\alpha,s}}\leq 1$. This yields that 
$$\frac{1}{C}\int w_j(x) dx \leq \|w_j\|_{\mathcal{L}^1(C)}\leq 1, \qquad \forall j\geq 1. $$

Thus by Koml\'os Theorem (see \cite{Kom}), one can find some subsequence of $\{ w_{j}\}$,  still denoted by $\{w_{j}\}$, and a function $w$ such that  $$\sigma_{k}(x):=\dfrac{1}{k}\sum_{j=1}^{k}w_{j}(x)\rightarrow w(x)$$ for almost everywhere $x$. Moreover, any subsequence of $\{w_j\}$ is also Ces\`aro convergent to $w$ almost everywhere. 
Then for any function $g$ such that $\|g\|_{M_1^{\alpha,s}}\leq 1$, by Fatou's lemma  we have

\begin{align}
\int |g(x)| w(x) dx&\leq \liminf_{k\rightarrow\infty} \int |g(x)|  \sigma_k(x) dx\notag\\
&=\liminf_{k\rightarrow\infty}\dfrac{1}{k}\sum_{j=1}^{k} \int |g(x)|  w_j(x) dx \notag\\
&\leq \liminf_{k\rightarrow\infty}\dfrac{1}{k}\sum_{j=1}^{k}\| w_{j}\|_{\mathcal{L}^{1}(C)}\notag\\
&\leq 1.\notag
\end{align}

This shows that $w\in \mathcal{L}^1(C)$ and $\|w\|_{\mathcal{L}^1(C)}\leq 1$. Furthermore, for each $j\geq 1$, by the convexity of the function 
$t \mapsto t^{1-q}$ on $(0,\infty)$, we have

\begin{align*}
\int f_{j}(x)^{q} w(x)^{1-q}dx& =  \int f_{j}(x)^{q}  \lim_{k\rightarrow\infty} \Big[\dfrac{1}{k}\sum_{m=j}^{j+k-1} w_{m}(x)\Big]^{1-q}dx\\
&\leq \liminf_{k\rightarrow\infty}\int f_{j}(x)^{q} \Big[\dfrac{1}{k}\sum_{m=j}^{j+k-1} w_{m}(x)\Big]^{1-q}dx\\
&\leq\liminf_{k\rightarrow\infty}\int f_{j}(x)^{q}\,  \dfrac{1}{k}\sum_{m=j}^{j+k-1} w_{m}(x)^{1-q}dx\\
&\leq\liminf_{k\rightarrow\infty}\int   \dfrac{1}{k}\sum_{m=j}^{j+k-1} f_{m}(x)^{q} \, w_{m}(x)^{1-q}dx,
\end{align*}
where we used $0\leq f_j\leq f_m$ for $m\geq j$ in the last bound. By \eqref{Mplus}, this gives
\begin{align*}
\int f_{j}(x)^{q} w(x)^{1-q}dx \leq (M+\epsilon)^q,
\end{align*}
and letting $j\rightarrow\infty$ we get
\begin{align*}
\int f(x)^{q} w(x)^{1-q}dx \leq (M+\epsilon)^q.
\end{align*}

As this holds for all $\epsilon>0$ we arrive at 
$$\|f\|_{\mathcal{N}^{\alpha,s}_q}\leq M=\sup_{j\geq 1}\|f_{j}\|_{\mathcal{N}^{\alpha,s}_q},$$
which completes the proof of the theorem.
\end{proof}

We now obtain the main result of this section.

\begin{theorem}  \label{NMprime} For $\alpha>0, s>1$ with $\alpha s\leq n$ and $p>1$ we have 
\begin{equation}\label{Nprime}
 \left(\mathcal{N}^{\alpha,s}_{p'}\right)' \approx M_{p}^{\alpha,s},
\end{equation}
and 
\begin{equation}\label{Mprime}
 \mathcal{N}^{\alpha,s}_{p'} \approx \left(M_{p}^{\alpha,s}\right)'.
\end{equation}
 \end{theorem}

\begin{proof} The relation \eqref{Nprime} is just a consequence of Theorem \ref{NmathcaloA1Th} and in fact more precisely we have 
$$\|f\|_{\left(\mathcal{N}^{\alpha,s}_{p'}\right)'}\leq \|f\|_{M_p^{\alpha,s}} \leq C  \|f\|_{\left(\mathcal{N}^{\alpha,s}_{p'}\right)'}.$$ 
 
To prove \eqref{Mprime} we note  from  \eqref{Nprime} that $\left(\mathcal{N}^{\alpha,s}_{p'}\right)'' \approx \left(M_{p}^{\alpha,s}\right)'$. On the other hand,  by Theorems \ref{XX''} and \ref{NisBFS} we find 
$\left(\mathcal{N}^{\alpha,s}_{p'}\right)''=\mathcal{N}^{\alpha,s}_{p'}$. Thus we obtain \eqref{Mprime} as claimed.
\end{proof}

\begin{remark}
If we drop the $A_1^{\rm loc}$ condition in the definition of 	$\mathcal{N}^{\alpha,s}_{q}$, then we get another space which we call $\widetilde{\mathcal{N}}^{\alpha,s}_{q}$. For this space we have $\left(\widetilde{\mathcal{N}}^{\alpha,s}_{p'}\right)' = M_{p}^{\alpha,s}$ and 
$\widetilde{\mathcal{N}}^{\alpha,s}_{p'} = \left(M_{p}^{\alpha,s}\right)'$, $p>1$. In particular, we have $\mathcal{N}^{\alpha,s}_{q} \approx \widetilde{\mathcal{N}}^{\alpha,s}_{q}$, $q>1$.
\end{remark}

\section{Proof of Theorems \ref{isomorphism} and \ref{isomorphism2}}

In order to prove Theorem \ref{isomorphism} we need the following lemmas.

\begin{lemma}\label{equivstrongsub} Suppose that $q>1$ and ${\rm Cap}_{\alpha,s}$ is strongly subadditive. Then $N^{\alpha,s}_q$ and $\widetilde{N}^{\alpha,s}_q$ are Banach spaces and we have 
$N^{\alpha,s}_q \approx \widetilde{N}^{\alpha,s}_q \approx B^{\alpha,s}_q$ with
\begin{equation*}
 \|f\|_{\widetilde{N}^{\alpha,s}_q} \leq \|f\|_{B^{\alpha,s}_q} \leq c_1 \,\|f\|_{\widetilde{N}^{\alpha,s}_q} \leq c_1 \,\|f\|_{N^{\alpha,s}_q} \leq c_2 \,\|f\|_{B^{\alpha,s}_q}.
\end{equation*}
\end{lemma}

\begin{proof} By reasoning as in the proof of Theorem \ref{NisBFS} we find
\begin{equation}\label{norm2}
\|f\|_{N^{\alpha,s}_q}= \inf\Big\{\sum_{j}|c_j|: f= \sum_{j} c_j b_j \text{ a.e.} \Big\},
\end{equation}
where each $b_j\in N^{\alpha,s}_q$ and $\|b_j\|_{N^{\alpha,s}_q}\leq 1$. Likewise,
\begin{equation}\label{norm3}
\|f\|_{\widetilde{N}^{\alpha,s}_q}= \inf \Big\{\sum_{j}|c_j|: f= \sum_{j} c_j b_j \text{ a.e.}\Big\},
\end{equation}
where each $b_j\in \widetilde{N}^{\alpha,s}_q$ and $\|b_j\|_{\widetilde{N}^{\alpha,s}_q}\leq 1$. Note here that to verify \eqref{norm2} we use the completeness of $L^1(C)$ (Theorem \ref{completeness}) to obtain that 
if $w=\sum_j |c_j| w_j$, where $\{c_j\}\in \ell^1$ and $w_j\in L^1(C)$ with $\|w_j\|_{L^1(C)}\leq 1$, then  $w$ is quasicontinuous  and $\|w_j\|_{L^1(C)}\leq \sum_j |c_j|$. Now \eqref{norm2} and \eqref{norm3} yield that 
 $N^{\alpha,s}_q$ and $\widetilde{N}^{\alpha,s}_q$ are Banach spaces.

Note that if $a\in L^{q}(\RR^n)$ is such that there exists a bounded set $A\subset\RR^n$
	for which $a=0$ a.e. in $\RR^n\setminus A$ and   $\|a\|_{L^{q}}\leq {\rm Cap}_{\alpha,s}(A)^{\frac{1-q}{q}}$, then obviously
$$\|a\|_{\widetilde{N}^{\alpha,s}_q} \leq 1$$
and by Lemma \ref{L1CforVEt}
$$\|a\|_{N^{\alpha,s}_q} \leq C.$$

Thus it follows from \eqref{norm2} and \eqref{norm3} that 
$$\|f\|_{\widetilde{N}^{\alpha,s}_q} \leq \|f\|_{B^{\alpha,s}_q}, \qquad \text{and~}\|f\|_{N^{\alpha,s}_q} \leq C\, \|f\|_{B^{\alpha,s}_q}.$$

Also, it is obvious that $\|f\|_{\widetilde{N}^{\alpha,s}_q} \leq \|f\|_{N^{\alpha,s}_q}$ and so we just need to show
\begin{equation}\label{BNtil}
\|f\|_{B^{\alpha,s}_q} \leq C \,\|f\|_{\widetilde{N}^{\alpha,s}_q}
\end{equation}
for any $f\in \widetilde{N}^{\alpha,s}_q$. Now for  $f\in \widetilde{N}^{\alpha,s}_q$, we can find a nonnegative function  $w$ defined quasieverywhere such that 
\begin{align*}
\int w\, dC=\displaystyle\int_{0}^{\infty}\text{Cap}_{\alpha,s}(\{w>t\})dt\leq 1 
\end{align*}
and 
\begin{align*}
\left(\int |f(x)|^{q} w(x)^{1-q}dx\right)^\frac{1}{q}\leq 2 \|f\|_{\widetilde{N}^{\alpha,s}_q}.
\end{align*}

Note that 
\begin{align*}
\sum_{k\in{{\ZZ}}}2^{k}&\text{Cap}_{\alpha,s}(\{2^{k-1}< w\leq 2^{k}\})\\
&=\frac{1}{4}\sum_{k\in{{\ZZ}}}\int_{2^{k-2}}^{2^{k-1}}\text{Cap}_{\alpha,s}(\{2^{k-1}<w\leq 2^{k}\})dt\\
&\leq \frac{1}{4}\sum_{k\in{\bf{Z}}}\int_{2^{k-2}}^{2^{k-1}}\text{Cap}_{\alpha,s}(\{w>t\})dt\\
&= \frac{1}{4} \int w\, dC \leq \frac{1}{4}.
\end{align*}

Let $E_{k}=\{2^{k-1}< w\leq 2^{k}\}$ for $k\in\ZZ$ and $D_l= \{l-1\leq  |x|<l\}$ for $l=1,2, \dots$ Note that $w<+\infty$ quasieverywhere  and hence 
\begin{align*}
f=\sum_{k,l}f\chi_{E_{k}\cap D_{l}}=\sum_{k,l}c_{k,l}a_{k,l} \qquad \text{a.e.},
\end{align*}
where  $\sum_{k,l} = \sum_{k\in \ZZ}\sum_{l\geq 1}$ and 
\begin{align*}
c_{k,l}&=\|f\|_{L^{q}(E_{k}\cap D_{l})}\text{Cap}_{\alpha,s}(E_{k}\cap D_{l})^{(q-1)/q},\\
a_{k,l}&=\|f\|_{L^{q}(E_{k}\cap D_{l})}^{-1}\text{Cap}_{\alpha,s}(E_{k}\cap D_{l})^{(1-q)/q}f\chi_{E_{k}\cap D_{l}}.
\end{align*}
Here we understand that $a_{k,l}=0$ whenever $f=0$ a.e. in $E_{k}\cap D_{l}$. 
It is obvious that 
\begin{align*}
\|a_{k,l}\|_{L^{q}}=\text{Cap}_{\alpha,s}(E_{k}\cap D_{l})^{(1-q)/q},
\end{align*}
whenever $a_{k,l}\not=0$. Moreover, we have   that
\begin{align*}
\sum_{k,l}c_{k,l}&=\sum_{k,l}\left(\int_{E_{k}\cap D_{l}}|f(x)|^{q} w(x)^{1-q} w(x)^{q-1}dx\right)^{\frac{1}{q}}\text{Cap}_{\alpha,s}(E_{k}\cap D_{l})^{\frac{q-1}{q}}\\
&\leq\sum_{k,l}\left(\int_{E_{k}\cap D_{l}}|f(x)|^{q}   w(x)^{1-q}dx\right)^{\frac{1}{q}}2^{k(q-1)/q}\text{Cap}_{\alpha,s}(E_{k}\cap D_{l})^{\frac{q-1}{q}  }\\
&\leq\Big(\sum_{k,l}\int_{E_{k}\cap D_{l}}|f(x)|^{q} w(x)^{1-q}dx\Big)^{\frac{1}{q}}\Big(\sum_{k,l}2^{k}\text{Cap}_{\alpha,s}(E_{k}\cap D_{l})\Big)^{\frac{q-1}{q}},
\end{align*}
where we  used H\"older's inequality in the last line.

On the other hand, it follows from the quasiadditivity of  ${\rm Cap}_{\alpha,s}$ (see \eqref{QAC}) we have 
$$\sum_{l\geq 1} \text{Cap}_{\alpha,s}(E_{k}\cap D_{l})\leq C\, \text{Cap}_{\alpha,s}(E_{k}).$$

Thus,

\begin{align*}
\sum_{k,l}c_{k,l}& \leq C \left(\int |f(x)|^{q} w(x)^{1-q}dx\right)^{\frac{1}{q}}\left(\sum_{k}2^{k}\text{Cap}_{\alpha,s}(E_{k})\right)^{\frac{q-1}{q} }\\
&\leq C \|f\|_{\widetilde{N}^{p',\alpha,s}}.
\end{align*}

 We have succeeded to decompose $f$ as the sum $f=\sum_{j}c_{j}a_{j}$
such that $\|{c_{j}}\|_{l^{1}}\leq C\|f\|_{\widetilde{N}^{p',\alpha,s}}$ and 
$\|a_{j}\|_{L^{q}}\leq \text{Cap}_{\alpha,s}(A_{j})^{(1-q)/q}$
with $\{a_{j}\ne 0\}\subset A_{j}$ for a bounded set $A_j$. Thus by the definition of $B^{\alpha,s}_{q}$ we obtain  $f\in B^{\alpha,s}_{q}$ with the bound \eqref{BNtil}.
\end{proof}

\begin{lemma}  \label{NNtilB} Suppose that $q>1$ and ${\rm Cap}_{\alpha,s}$ is strongly subadditive. Then $N^{\alpha,s}_q$, $\widetilde{N}^{\alpha,s}_q$, and $B^{\alpha,s}_q$ are Banach function spaces.
\end{lemma}	

\begin{proof} First we show that $\widetilde{N}^{\alpha,s}_q$ is a Banach function space and for that we just need to check the Fatou property (P2).
The proof is similar to that of Theorem \ref{NisBFS}. 	
Let $\{f_j\}$, $j=1,2, \dots$, be a sequence of nonnegative measurable functions in $\widetilde{N}^{\alpha,s}_{q}$ and $f_j \uparrow f$ a.e. in $\RR^n$. Suppose that 
$\sup_{j\geq 1} \|f_{j}\|_{\widetilde{N}^{\alpha,s}_{q}}=M<+\infty$. It is enough to show that $\|f\|_{\widetilde{N}^{\alpha,s}_{q}} \leq M.$

As in the proof of Theorem \ref{NisBFS},  for any $j\geq 1$ and $\epsilon>0$ we can find  a nonnegative and q.e. defined weight $w_{j}$ with  $\int w_j\, dC\leq 1$ 
such that  
\begin{align*}
\left(\int|f_{j}(x)|^{q} w_{j}(x)^{1-q}dx\right)^{\frac{1}{q}} \leq M+\epsilon
\end{align*}
and $\int w_j(x) dx \leq C.$ Then by Koml\'os Theorem, one can find a subsequence of $\{ w_{j}\}$,  still denoted by $\{w_{j}\}$, and a function $w$ such that  $\sigma_{k}(x):=\dfrac{1}{k}\sum_{j=1}^{k}w_{j}(x)\rightarrow w(x)$ for almost everywhere $x$. Moreover, any subsequence of $\{w_j\}$ is also Ces\`aro convergent to $w$ almost everywhere.
By redefining $w(x)$ to be zero for all the points $x$ such that  $\sigma_{k}(x)\not\rightarrow w(x)$, one has 
\begin{align*}
w(x)\leq\liminf_{k\rightarrow\infty}\sigma_{k}(x) \qquad \text{quasieverywhere.}
\end{align*} 

Hence,
\begin{align}
\int_{0}^{\infty}\text{Cap}_{\alpha,s}\left(\left\{ w>t\right\}\right)dt   &\leq\int_{0}^{\infty}\text{Cap}_{\alpha,s}\left(\left\{\liminf_{k\rightarrow\infty}\sigma_{k}>t\right\}\right)dt\notag\\
&\leq\int_{0}^{\infty}\liminf_{k\rightarrow\infty}\text{Cap}_{\alpha,s}(\{\sigma_{k}>t\})dt\notag\\
&\leq\liminf_{k\rightarrow\infty}\int_{0}^{\infty}\text{Cap}_{\alpha,s}(\{\sigma_{k}>t\})dt\notag\\
&\leq\liminf_{k\rightarrow\infty}\dfrac{1}{k}\sum_{j=1}^{k}\int_{0}^{\infty}\text{Cap}_{\alpha,s}(\{w_{j}>t\})dt,\notag
\end{align}
where we  used the strong subadditivity of ${\rm Cap}_{\alpha,s}$ in the last inequality. This gives
$$
\int w\, dC   \leq \liminf_{k\rightarrow\infty}\dfrac{1}{k}\sum_{j=1}^{k} \int w_{j} \, dC \leq 1.
$$

Moreover, as in the proof of Theorem \ref{NisBFS} we have 
$$
\int|f(x)|^{q} w(x)^{1-q}dx \leq (M+\epsilon)^q
$$
and so $\|f\|_{\widetilde{N}^{\alpha,s}_q}\leq M$ as desired.

Next we show that $N^{\alpha,s}_q$ is a Banach function space. Let $\{f_j\}$, $j=1,2, \dots$, be a sequence of nonnegative measurable functions in $N^{\alpha,s}_{q}$ and $f_j \uparrow f$ a.e. in $\RR^n$. Suppose that 
$\sup_{j\geq 1} \|f_{j}\|_{N^{\alpha,s}_{q}}=M<+\infty$. By Lemma \ref{equivstrongsub} and the Fatou property of $\widetilde{N}^{\alpha,s}_{q}$, we have 
$\|f\|_{\widetilde{N}^{\alpha,s}_{q}}\leq M$. In particular, $f\in \widetilde{N}^{\alpha,s}_{q}\cap N^{\alpha,s}_{q}$ and thus $g_j:=f-f_j \in \widetilde{N}^{\alpha,s}_{q}\cap N^{\alpha,s}_{q}$ for all $j\geq 1$ and 
$g_j\downarrow 0$ a.e.

On the other hand, Theorem \ref{second} implies that   $\left(\widetilde{N}^{\alpha,s}_{q}\right)^*=\left(\widetilde{N}^{\alpha,s}_{q}\right)'$ with equality of norms and hence 
it follows from Theorem \ref{X*X'} that $\widetilde{N}^{\alpha,s}_{q}$ has an absolutely continuous norm. This yields that $\|f-f_j\|_{\widetilde{N}^{\alpha,s}_{q}}=\|g_j\|_{\widetilde{N}^{\alpha,s}_{q}}\downarrow 0$.
Thus by Lemma \ref{equivstrongsub} we then obtain $\|g_j\|_{N^{\alpha,s}_{q}}\downarrow 0$.
This yields $\|f_j\|_{N^{\alpha,s}_{q}}\uparrow \|f\|_{N^{\alpha,s}_{q}}$ and the Fatou property (P2) follows for $N^{\alpha,s}_{q}$.
It is now easy to see that $N^{\alpha,s}_q$ is a Banach function space.

The proof that $B^{\alpha,s}_q$ is a Banach function space can be proceeded similarly, as long as we can verify the following  properties of $B^{\alpha,s}_q$:

\begin{equation}\label{absol}
 \|f\|_{B^{\alpha,s}_q} = \| |f|\|_{B^{\alpha,s}_q}, \qquad \forall f\in B^{\alpha,s}_q,
\end{equation}
and
\begin{equation}\label{lat}
0\leq f\leq g  \text{~a.e.~} \Rightarrow  \|f\|_{B^{\alpha,s}_q}\leq \|g\|_{B^{\alpha,s}_q}, \qquad \forall g\in B^{\alpha,s}_q.
\end{equation}

Equality  \eqref{absol} is easy to see from the identities $|f|=f{\rm sgn}(f)$ and $f=|f|{\rm sgn}(f)$.
To see \eqref{lat}, suppose that  $g\in B^{\alpha,s}_q$ and $g=\sum_j c_j a_j$, where  $\{c_{j}\}\in \ell^{1}$ and each $a_j\in L^{q}(\RR^n)$ is such that there exists a bounded set $A_j\subset\RR^n$
for which $a_{j}=0$ a.e. in $\RR^n\setminus A_j$ and  $\|a_{j}\|_{L^{q}}\leq {\rm Cap}_{\alpha,s}(A_{j})^{\frac{1-q}{q}}$. Then for $0\leq f\leq g$ we can write
$$f=\sum_j c_j  f g^{-1} a_j \chi_{\{g\not=0\}} \qquad \text{a.e.}$$
Note that $\|f g^{-1} a_j \chi_{\{g\not=0\}}\|_{L^q} \leq \| a_j\|_{L^q} \leq {\rm Cap}_{\alpha,s}(A_{j})^{\frac{1-q}{q}}$, and thus $f\in B^{\alpha,s}_q$ and 
$$\|f\|_{B^{\alpha,s}_q}\leq \|g\|_{B^{\alpha,s}_q}.$$

This completes the proof of the lemma.
\end{proof}	

We can now prove Theorem \ref{isomorphism}.

\begin{proof}[Proof of Theorem \ref{isomorphism}] 
By Theorems  \ref{NoA1Th} and \ref{BspaceTh}  we have 
	$$\left(M^{\alpha, s}_p\right)'= \left(\widetilde{N}^{\alpha, s}_{p'}\right)'' =  \left(B^{\alpha, s}_{p'}\right)''.$$	 
	
Thus if  ${\rm Cap}_{\alpha,s}$ is strongly subadditive then by Theorem \ref{XX''} and Lemma \ref{NNtilB}	we find
$$ (M^{\alpha, s}_p)' = \widetilde{N}^{\alpha, s}_{p'} =  B^{\alpha, s}_{p'}.$$	 

Likewise, by Theorem \ref{second} we have $\left(N^{\alpha, s}_{p'}\right)'' \approx (M^{\alpha, s}_p)'$ and so it follows from Theorem \ref{XX''} and Lemma \ref{NNtilB}
that $N^{\alpha, s}_{p'} \approx (M^{\alpha, s}_p)'$.	 

Now the theorem follows from Theorems \ref{NisBFS}, \ref{NMprime}, and Lemma  \ref{NNtilB}.
\end{proof}


We next prove Theorem \ref{isomorphism2}.

\begin{proof}[Proof of Theorem \ref{isomorphism2}]

We will need the following functional  which is defined by
$$\gamma_{\alpha,s}(u):=\inf\left\{\int f^sdx: 0\leq f\in L^s(\RR^n) \text{ and } G_\alpha*f\geq |u|^{\frac{1}{s}} \text{ q.e.} \right\}$$
for each q.e. defined function $u$ in $\RR^n$. Note that $\gamma_{\alpha,s}(tu)=|t|\gamma_{\alpha,s}(u)$ for all $t\in\RR$, and moreover,
$$\gamma_{\alpha,s}(u):=\inf\left\{t>0: |u|\in t\, H \right\},$$
where $H$ is the set of all nonnegative and q.e. defined functions $g$ in $\RR^n$ such that $g^{\frac{1}{s}}\leq G_\alpha *f$ q.e. for a function 
$f\in L^s(\RR^n)$, $f\geq 0$, such that $\norm{f}_{L^s}\leq 1$.
Note that if we define a nonlinear operator $T$ by $$T(h):= \left(G_\alpha*|h|^{\frac{1}{s}}\right)^{s}, \qquad \forall h\in L^1(\RR^n),$$
then by reverse Minkowski's inequality we see that $T$ is superadditive. 
This yields that the set  $H$ is convex (see \cite[Lemma 2.4]{KV}) and thus the functional $\gamma_{\alpha,s}(\cdot)$ is subadditive. 

On the other hand, we can  deduce from \cite[Proposition 7.4.1]{AH} the following equivalence 
\begin{equation}\label{funequ}
\int_0^\infty {\rm Cap}_{\alpha,s}(\{x: |u(x)|>t\}) dt \simeq \gamma_{\alpha,s}(u),
\end{equation}
which holds for all q.e. defined functions $u$ in $\RR^n$. In particular, we find that the space $L^1(C)$ is normable for all $\alpha>0, s>1$ and $\alpha s\leq n$.

Now using \eqref{funequ} and the subadditivity of $\gamma_{\alpha,s}(\cdot)$, and arguing as in the proof of Theorem \ref{NisBFS} we find
\begin{equation*}
\|f\|_{N^{\alpha,s}_q} \simeq \inf\Big\{\sum_{j}|c_j|: f= \sum_{j} c_j b_j \text{ a.e. where } \|b_j\|_{N^{\alpha,s}_q}\leq 1 \Big\},
\end{equation*}
 and
\begin{equation*}
\|f\|_{\widetilde{N}^{\alpha,s}_q} \simeq \inf\Big\{\sum_{j}|c_j|: f= \sum_{j} c_j b_j \text{ a.e. where } \|b_j\|_{\widetilde{N}^{\alpha,s}_q}\leq 1\Big\}.
\end{equation*}

At this point we can repeat the argument in the proof of Lemma  \ref{equivstrongsub} to obtain 
$N^{\alpha,s}_q \approx \widetilde{N}^{\alpha,s}_q \approx B^{\alpha,s}_q$. Combining this with Theorem  \ref{NMprime} we get the theorem.
\end{proof}

\section{Proof of Theorem \ref{triplettheorem}}

The proof of Theorem \ref{triplettheorem} is based on Hahn-Banach Theorem and the following lemma. The Morrey space version of this lemma can be found in \cite[Theorem 4.3]{ST}.

\begin{lemma}\label{finite}
Suppose that $p>1$, $\alpha>0, s>1$, and $\alpha s \leq n$. Let $\mathcal{F}$ be the  closer in $M^{\alpha,s}_p$ of the set of all finite linear combinations of the characteristic functions of sets of finite Lebesgue measure. 
Then $\mathcal{F}^{*} \approx \mathcal{N}^{\alpha,s}_{p'}$ in the sense that each linear functional   $L\in \mathcal{F}^{*}$ corresponds to a unique $g\in \mathcal{N}^{\alpha,s}_{p'}$ in such a way that 
\begin{align*}
L(f)= \int f(x)g(x)dx, \qquad \forall f\in \mathcal{F},  
\end{align*}
and $\|L\| \simeq |g\|_{\mathcal{N}^{\alpha,s}_{p'}}$.

\end{lemma}

\begin{proof} If $E$ is a measurable set with finite measure then by \eqref{aslessn} and \eqref{asequaln} we have 
\begin{equation*}
\|\chi_{E}\|_{M^{\alpha,s}_{p}} \leq C |E|^\frac{\alpha s} {np} \text{~for~} \alpha s <n,
\end{equation*}
and
\begin{equation*}
\|\chi_{E}\|_{M^{\alpha,s}_{p}} \leq C |E|^{\frac{1}{2p}} \text{~for~} \alpha s =n.
\end{equation*}

Thus $\mathcal{F}$ is a nonempty closed subspace of $M^{\alpha,s}_{p}$, and 
\begin{align*}
\|\chi_{A}-\chi_{B}\|_{M^{\alpha,s}_{p}}&=\|\chi_{(A-B)}\|_{M^{\alpha,s}_{p}}\leq  C\, |A-B|^{\kappa}
\end{align*}
for any sets $B \subset A\subset\RR^n$. Here $\kappa=\alpha s/(np)$ if $\alpha s<n$ and $\kappa=1/(2p)$ if $\alpha s=n$. 

The lemma then follows from Radon-Nikodym Theorem and relation \eqref{Mprime} via the same argument as in the proof of \cite[Theorem  4.3]{ST}.
\end{proof}

\begin{proof}[Proof of Theorem \ref{triplettheorem}] It is obvious that the embedding $$\mathcal{N}^{\alpha,s}_{p'} \hookrightarrow \left(\mathring{M}^{\alpha,s}_{p}\right)^*$$ is continuous.
To show the converse, we first show that $\mathring{M}^{\alpha,s}_{p}$ is a subspace of $\mathcal{F}$, where $\mathcal{F}$ is as in Lemma \ref{finite}. To see this, let 
$v\in C_c$ and  let $Q$ be a cube that contains the support of $v$. For $i=1,2,\dots,$ partition $Q$ into $2^{n i}$ smaller disjoint cubes, say, $Q_{1},...,Q_{2^{n i}}$, and let 
\begin{align*}
v_{i}=\sum_{k=1}^{2^{n i}} v(x_k)\chi_{Q_{k}}.
\end{align*} 
Here $x_k$ is the center of $Q_k$, $k=1,\dots, 2^{n i}$. As $v$ is uniformly continuous we see that  $v_{i}\rightarrow v$ in  $L^\infty(\RR^n)$ and hence $v_{i}\rightarrow v$ in  $M^{\alpha,s}_{p}$ by, e.g., \eqref{LinftyM}.
Thus $v\in \mathcal{F}$ and the claim follows.

Now let $L$ be a bounded linear functional on  $\mathring{M}^{\alpha,s}_{p}$. By Hahn-Banach Theorem there exists bounded linear functional $\tilde{L}$ on $\mathcal{F}$ such that $\tilde{L}(v)=L(v)$ for all $v\in \mathring{M}^{\alpha,s}_{p}$
and $\|\tilde{L}\|=\|L\|$. Lemma \ref{finite} implies that there exists a function $g\in \mathcal{N}^{\alpha,s}_{p'}$ such that 
\begin{align*}
\tilde{L}(f)= \int f(x)g(x)dx, \qquad \forall f\in \mathcal{F},  
\end{align*}
and $\|\tilde{L}\| \simeq \|g\|_{\mathcal{N}^{\alpha,s}_{p'}}$. Thus we have 
\begin{align*}
L(v)= \int v(x)g(x)dx, \qquad \forall v\in \mathring{M}^{\alpha,s}_{p},  
\end{align*}
and $\|L\| \simeq \|g\|_{\mathcal{N}^{\alpha,s}_{p'}}$. Note that $g$ is unique since if $\int v gdx =\int v \tilde{g} dx$  for all $v\in C_c$ then it holds that $g=\tilde{g}$ a.e.
\end{proof}

\section{Proof of Theorems \ref{Mlocbound} and \ref{CZop}}

To prove Theorems \ref{Mlocbound} and \ref{CZop}, we need the following basic results about $A_1^{\rm loc}$ weights.
We first observe that if $w\in A_1^{\rm loc}$ then for any ball $B_r$ with radius $r\leq 1/2$ we have 
\begin{equation}\label{A1locinf}
\fint_{B_r} w(y) dy \leq 2^n [w]_{A_1^{\rm loc}}\,  \inf_{B_r} w.
\end{equation}

Let $B$ be any ball such that the radius of $B$ is $r(B)=1/2$. We claim that there is a constant $c(n)>0$ such that  
\begin{equation}\label{A1doubling}
\int_{(t+1/2)B} w(y) dy \leq c(n) [w]_{A_1^{\rm loc}} \int_{t B} w(y)dy, \qquad \forall t\geq 1. 
\end{equation}

Indeed, for any $t\geq 1$, let $A=(t+1/2)B\setminus t B$. We can cover $A$ by balls $B_k$ of radius $r(B_k)=1/2$ such that $|B_k\cap tB| \geq c(n)$, and 
$$\sum_k \chi_{B_k}(x)\leq N(n).$$ 

Thus using \eqref{A1locinf} we have 
\begin{align*}
\int_{A} w(y)dy &\leq \sum_{k}  \int_{B_k} w(y)dy\leq c [w]_{A_1^{\rm loc}} \sum_k \inf_{B_k} w\\
&\leq c [w]_{A_1^{\rm loc}} \sum_k \inf_{B_k\cap tB} w \leq c [w]_{A_1^{\rm loc}} \sum_k \fint_{B_k\cap tB} w(y) dy\\
&\leq c [w]_{A_1^{\rm loc}}  \int_{tB} w(y) \sum_k \chi_{B_k}(y) dy \leq c [w]_{A_1^{\rm loc}}  \int_{tB} w(y) dy.
\end{align*}

It follows that 
$$\int_{(t+1/2)B} w(y)dy\leq (c [w]_{A_1^{\rm loc}}+1) \int_{t B} w(y)dy,$$\
and the claim \eqref{A1doubling} follows.

Using \eqref{A1doubling} we see that if $w\in A_1^{\rm loc}$ with $[w]_{A_1^{\rm loc}}\leq {\bf \overline{c}}$, then for any ball $B_r$ with radius $r\leq R_0$, $R_0>0$, we have 
\begin{equation}\label{A1locinfR0}
\fint_{B_r} w(y) dy \leq C(n, R_0, {\bf \overline{c}}) \, \inf_{B_r} w.
\end{equation}

Now using \eqref{A1locinfR0} and  a minor modification of the proof of \cite[Lemma 1.1]{Ryc}, we obtain the following result. 

\begin{lemma}\label{weightextend} Let $w\in A_1^{\rm loc}$ with $[w]_{A_1^{\rm loc}}\leq {\bf \overline{c}}$ and $B=B_{R_0}(x_0)$, $R_0>0$. Then there exists a weight $\overline{w}\in A_1$ such that $w=\overline{w}$ in $B$ and $[\overline{w}]_{A_1}\leq c(n, R_0, {\bf \overline{c}})$. 
\end{lemma}

As a consequence of \cite[Lemma 2.11]{Ryc} and \eqref{A1locinfR0}, we also have the following weighted bound for ${\bf M}^{\rm loc}$.
\begin{lemma}\label{Mlocweightedbound} For any $p>1$ and $w\in A_1^{\rm loc}$ with $[w]_{A_1^{\rm loc}}\leq {\bf \overline{c}}$, it holds that 
	$$\int_{\RR^n} {\bf M}^{\rm loc}(f)^{p} w dx \leq C(n,p, {\bf \overline{c}}) \int_{\RR^n} |f|^p w dx.$$
\end{lemma}

	In fact, Lemma \ref{Mlocweightedbound}  also holds for  $w$ in the larger class of $A_p^{\rm loc}$ weights (see \cite[Lemma 2.11]{Ryc}).

We are now ready to prove Theorems \ref{Mlocbound} and \ref{CZop}.

\begin{proof}[Proof of Theorem \ref{Mlocbound}] This theorem follows from  Theorem \ref{isomorphism2} and Lemma \ref{Mlocweightedbound}.
\end{proof}	

\begin{proof}[Proof of Theorem \ref{CZop}] Let $w\in A_1^{\rm loc}$ with $[w]_{A_1^{\rm loc}}\leq {\bf \overline{c}}$ and  suppose that ${\rm supp}(f)\subset B_{R_0}(x_0)$ for some $ R_0>0$. 
By Lemma \ref{weightextend}, there exists a weight $\overline{w}\in A_1$ with $[\overline{w}]_{A_1}\leq C(n, R_0, {\bf \overline{c}})$ such that 
$\overline{w}=w$ in $B_{R_0}(x_0)$.	Thus it follows from the hypothesis of the theorem that 
\begin{align*}
\int |T(f)|^{q}\chi_{B_{R_0}(x_0)} w dx & \leq  \int |T(f)|^{q} \overline{w} dx\\
& \leq C(n,q,{\bf \overline{c}}) \int |f|^{q} \overline{w} dx\\
&= C(n,q,{\bf \overline{c}}) \int |f|^{q} w dx.
\end{align*}

Theorem \ref{CZop} now follows from Theorems \ref{isomorphism2} and \ref{firsttheorem}.
\end{proof}	

\section{The homogeneous case}\label{homogeneoussetting}

Let $p\geq 1$, $\alpha >0$,  and $s>1$ be such that $\alpha s <n$. The homogeneous version of $M^{\alpha,s}_p$, denoted as $\dot{M}^{\alpha,s}_p$, is 
the space of functions $f\in L^p_{\rm loc}(\mathbb{R}^n)$
such that the trace inequality 
\begin{align}\label{trace-homo}
	\left(\int_{\mathbb{R}^n} |u|^s |f|^p dx\right)^{\frac{1}{p}} \leq C \|u\|_{\dot{H}^{\alpha,s}}^{\frac{s}{p}}
\end{align} 
holds for all $u\in C_c^\infty(\mathbb{R}^n)$. A norm  of a function $f\in \dot{M}^{\alpha,s}_p$  is defined as the least possible constant $C$ in the above inequality. In \eqref{trace-homo}, $\dot{H}^{\alpha,s}$ stands for 
the space of Riesz potentials which consists of functions of the form $u= I_{\alpha}* f$ for some $f\in L^s(\RR^n)$, and 
$\|u\|_{\dot{H}^{\alpha, s}}=\|f\|_{L^s}.$
 Here $I_{\alpha}$, $\alpha\in (0,n)$, is the Riesz kernel defined as the inverse Fourier transform of $|\xi|^{\alpha}$ (in the distributional sense), and explicitly we have 
$I_\alpha(x)= \gamma(n,\alpha) |x|^{\alpha-n}$, where $\gamma(n, \alpha)=\Gamma(\tfrac{n-\alpha}{2})/[\pi^{n/2} 2^\alpha \Gamma(\tfrac{\alpha}{2})]$. It is known that (see \cite{MH}) $\dot{H}^{\alpha,s}$ is the completion of $C_c^\infty(\mathbb{R}^n)$ with respect to the norm 
$$\|u\|_{\dot{H}^{\alpha,s}}=\|(-\Delta)^{\frac{\alpha}{2}}u \|_{L^s(\mathbb{R}^n)}= \| \mathcal{F}^{-1}(|\xi|^{\alpha}\mathcal{F}(u))\|_{L^s(\RR^n)} $$

In the case  $\alpha =k\in \mathbb{N}$ and $s>1$  we have $\dot{H}^{k,s} \approx \dot{W}^{k,s}$, where $\dot{W}^{k,s}=\dot{W}^{k,s}(\mathbb{R}^n)$ is the homogeneous Sobolev space with norm being defined as
\begin{align*}
	\|u\|_{\dot{W}^{k,s}}=\sum_{|\beta|= k}\|D^{\beta}u\|_{L^{s}}.
\end{align*}

The capacity associated to $\dot{H}^{\alpha,s}$ is the Riesz capacity defined for each set $E\subset\RR^n$ by 
\begin{equation*}
{\rm cap}_{\alpha, \,s}(E):=\inf\Big\{\|f\|_{L^{s}}^{s}: f\geq0, I_{\alpha}*f\geq 1 {\rm ~on~} E \Big\}.
\end{equation*}

It is known that   the norm of a function $f\in \dot{M}^{\alpha,s}_p$ is equivalent to the quantity
$$\sup_{K}\left(\frac{\int_{K}|f(x)|^{p}dx}{\text{cap}_{\alpha,s}(K)}\right)^{1/p}, $$
where the supremum is taken over all compact sets $K\subset\RR^n$ (see \cite{MS2, AH}). For this reason, we shall use this quantity as the norm for $\dot{M}^{\alpha,s}_p$, $p\geq 1$, in what follows.

For $s=2$ and $\alpha\in (0,1]$, it is known that  ${\rm cap}_{\alpha,s}(\cdot)$ is strongly subadditive (see \cite[pp. 141-145]{Lan}).
On the other hand, for $\alpha=1$, ${\rm cap}_{1,s}(\cdot)$ is equivalent to the capacity $c_{1, s}(\cdot)$, which is a strongly subadditive capacity (see \cite[Theorem 2.2]{HKM}).
Here for each compact set $K\subset\RR^n$, 
$$c_{1, s}(K):=\inf\left\{  \int |\nabla u|^s dx: \varphi\in C_c^\infty, \varphi\geq 1 \text{~on~} K \right\},$$
and $c_{1, s}(\cdot)$ is extended to all sets $E$ as in \eqref{extension}. 

Note that for all balls $B_r$, $r>0$, we have 
\begin{equation*}
{\rm cap}_{\alpha,s}(B_r)\simeq r^{n-\alpha s},
\end{equation*}
and  for any measurable  set $E$ it follows from the Sobolev Embedding Theorem that 
\begin{equation*}
|E|^{1-\alpha s/n} \leq C\, {\rm cap}_{\alpha,s}(E).
\end{equation*}

This lower bound implies that  $L^{\frac{n}{\alpha s},\infty}(\RR^n)\hookrightarrow \dot{M}^{\alpha, s}_p$ with the estimate
\begin{equation}\label{LweakM-homo}
\|f\|_{\dot{M}^{\alpha, s}_p}\leq C \|f\|_{L^{\frac{n p}{\alpha s},\infty}(\RR^n)}.
\end{equation}

The capacity ${\rm cap}_{\alpha,s}$ is quasiadditive in the following sense. There exists a constant $C=C(n,\alpha,s)>0$ such that for any set $E$ we have (see \cite[Eq.(7)]{Ad3} and \cite{Ad2})
$$\sum_{j=-\infty}^{\infty} {\rm cap}_{\alpha,s}(E \cap \{2^{j-1} \leq |x|< 2^j \}) \leq C\, {\rm cap}_{\alpha,s}(E).$$

The homogeneous version of $L^1(C)$ is $\dot{L}^1(C)$ which is defined analogously using the Riesz capacity ${\rm cap}_{\alpha,s}$. Likewise, the homogeneous version of $\mathcal{L}^1(C)$ is $\mathcal{\dot{L}}^1(C)$ which consists of 
measurable functions $w$ such that 
$$\|w\|_{\mathcal{\dot{L}}^1(C)} := \sup_{g}\int |g(x)| |w(x)| dx <+\infty, $$ 
where the supremum is taken over all $g\in \dot{M}^{\alpha,s}_1$ such that $\norm{g}_{\dot{M}^{\alpha,s}_1}\leq 1$.  That is, $\mathcal{\dot{L}}^1(C)$ is the K\"othe dual of $\dot{M}^{\alpha,s}_1$.
It is easy to see that the quasinormed space  $\dot{L}^1(C)$ is continuously embedded into the Banach space $\mathcal{\dot{L}}^1(C)$.

Let  $E$  be a subset of $\RR^n$ such that  $0<\text{cap}_{\alpha,s}(E)<\infty$. By \cite[Theorems 2.5.6 and 2.6.3 ]{AH}, the capacitary measure for $E$
exists as a  nonnegative measure $\mu^{E}$ with ${\rm supp}(\mu^{E}) \subset \overline{E}$  such that the function 
 $V^{E}=I_\alpha*((I_\alpha*\mu)^{\frac{1}{s-1}})$ satisfies the following properties:

\begin{equation*} 
\mu^E(\overline{E})={\rm Cap}_{\alpha,s}(E)=\int_{\RR^n} V^E d\mu^E= \int_{\RR^n} (I_\alpha *\mu^{E})^{\frac{s}{s-1}} dx,
\end{equation*}
\begin{equation*}
V^E\geq 1 \quad \text{quasieverywhere on }E, \quad {\rm and} \quad V^E\leq A \quad \text{on } \RR^n.
\end{equation*}

We have the following homogeneous version of Lemma \ref{L1CforVEt}.

\begin{lemma}
	 Let $E$, $\mu^E$, and $V^E$ be as above with $0<\alpha s < n$. If  $\delta\in (1, n/(n-\alpha))$ for $s< 2$ and $\delta\in (s-1, n(s-1)/(n-\alpha s))$ 
	for $s\geq 2$, then the function $(V^E)^\delta \in A_1$ with $[(V^E)^\delta]_{ A_1}\leq c(n, \alpha, s, \delta)$. Moreover, $(V^E)^\delta\in \dot{L}^1(C)$ with 
	$\|(V^E)^\delta\|_{\dot{L}^1(C)} \leq C\, {\rm cap}_{\alpha, s}(E)$. 	
\end{lemma}

This lemma follows from Lemmas 2 and 3 of \cite[Sub-section 2.6.3]{MS1}.  The following `renorming' theorem for  $\dot{M}^{\alpha,s}_p$ can be proved as in the inhomogeneous setting.
\begin{theorem}
	For $p>1$ and $\alpha>0, s>1$, with $\alpha s < n$, we have 
	\begin{align}\label{dotM-norm}
	\|f\|_{\dot{M}^{\alpha,s}_p} \simeq \sup_{w}\left(\int_{{\RR}^{n}}|f(x)|^{p} w(x)dx\right)^{1/p},
	\end{align}
	where the supremum is taken over all nonnegative  $w\in \dot{L}^{1}(C)\cap A_1$ with $\|w\|_{\dot{L}^{1}(C)}\leq 1$ and $[w]_{A_1}\leq {\bf \overline{c}}(n,\alpha,s)$ for a constant ${\bf \overline{c}}(n,\alpha,s)\geq 1$.  The equivalence \eqref{dotM-norm} 
also holds if we replace $\dot{L}^{1}$ by $\mathcal{\dot{L}}^{1}$. Moreover, we have 
\begin{align*}
	\|f\|_{\dot{M}^{\alpha,s}_p} = \sup_{w}\left(\int_{\RR^{n}}|f(x)|^{p} w(x)dx\right)^{1/p},
	\end{align*}
where the supremum is taken over all weights  $w$ such that $w$ is defined ${\rm cap}_{\alpha,s}$-quasieverywhere and 
$\int_{0}^\infty {\rm cap}_{\alpha,s}(\{x\in \RR^n: w(x)>t\}) dt \leq 1$.
\end{theorem}

The homogeneous versions of $N_q^{\alpha,s}, \mathcal{N}_q^{\alpha,s}, \widetilde{N}_q^{\alpha,s}$, and $B_q^{\alpha,s}$ are denoted by $\dot{N}_q^{\alpha,s}, \mathcal{\dot{N}}_q^{\alpha,s}, \dot{\widetilde{N}}_q^{\alpha,s}$, and $\dot{B}_q^{\alpha,s}$, respectively. They are defined similarly using ${\rm cap}_{\alpha, s}$ in place of 
${\rm Cap}_{\alpha, s}$. We have the following relations.

\begin{theorem}
	Let $p>1$ and $\alpha>0, s>1$, with $\alpha s < n$. Then
\begin{equation}\label{dualre}
\left(\dot{N}_{p'}^{\alpha,s}\right)^* \approx  \left(\mathcal{\dot{N}}_{p'}^{\alpha,s}\right)^*  \approx \dot{M}^{\alpha,s}_p=\left(\dot{\widetilde{N}}_{p'}^{\alpha,s}\right)^*= \left(\dot{B}_{p'}^{\alpha,s}\right)^*=\left[(\dot{M}^{\alpha,s}_p)'\right]^{*}.
\end{equation}

Moreover, the spaces $\mathcal{\dot{N}}^{\alpha, s}_{p'}$  and $(\dot{M}^{\alpha, s}_p)'$ are Banach function spaces  and $\mathcal{\dot{N}}^{\alpha, s}_{p'} \approx (\dot{M}^{\alpha, s}_p)'$.	Additionally, if 
 ${\rm cap}_{\alpha, s}$ is strongly subadditive then  $\dot{N}^{\alpha, s}_{p'}$,  $\dot{\widetilde{N}}^{\alpha, s}_{p'}$, and $\dot{B}^{\alpha, s}_{p'}$ are also Banach function spaces, and
\begin{equation*}
\mathcal{\dot{N}}^{\alpha, s}_{p'} \approx  \dot{N}^{\alpha, s}_{p'} \approx (\dot{M}^{\alpha, s}_p)' = \dot{\widetilde{N}}^{\alpha, s}_{p'} =  \dot{B}^{\alpha, s}_{p'}.
\end{equation*}

In general, if no strong subadditivity is assumed on ${\rm cap}_{\alpha, s}$ then we have 
\begin{equation}\label{chain-homo}
\mathcal{\dot{N}}^{\alpha, s}_{p'} \approx  \dot{N}^{\alpha, s}_{p'} \approx (\dot{M}^{\alpha, s}_p)' \approx \dot{\widetilde{N}}^{\alpha, s}_{p'} \approx  \dot{B}^{\alpha, s}_{p'}.
\end{equation}

\end{theorem}

Note that by \eqref{LweakM-homo} and \eqref{dualre}, all  spaces in \eqref{chain-homo} are continuously embedded into the Lorentz space $L^{\frac{np}{np-\alpha s},1}$.

The homogeneous version of Theorem \ref{triplettheorem} reads as follows.

\begin{theorem}
	Let $p>1$, $\alpha>0$, $s>1$, with $\alpha s < n$. 	We have 
	$$(\mathring{\dot{M}}^{\alpha,s}_{p})^{*} \approx \mathcal{\dot{N}}^{\alpha,s}_{p'},$$
	where we define $\mathring{\dot{M}}_{p}^{\alpha,s}$ as the closure of $C_c$ in $\dot{M}^{\alpha,s}_{p}$.
\end{theorem}

It is known that the Hardy-Littlewood maximal function ${\bf M}$ and standard  Calder\'on-Zygmund operators are bounded on $\dot{M}^{\alpha,s}_p$ (see \cite{MV} and \cite{MS1}). For other spaces,  we have the following results.

\begin{theorem} Let $q>1$, $\alpha>0$, $s>1$, and $\alpha s < n$. Suppose that $T$ is an operator (not necessarily linear or sublinear) such that 
	$$\int |T (f)|^{q} w dx \leq C_1 \int |f|^{q} w dx$$
holds for all $f\in L^{q}(w)$ and all $w\in A_1$, with a constant $C_1$ depending only on $n, q$, and the bound for the $A_1$ constant of $w$.	
Then $T$ is bounded on  $\dot{N}^{\alpha, s}_{q}, \mathcal{\dot{N}}^{\alpha, s}_{q}$, $(\dot{M}_{q'}^{\alpha ,s})'$, $\dot{\widetilde{N}}^{\alpha, s}_{q}$, and  $\dot{B}_{q}^{\alpha, s}$.
\end{theorem}

\vspace{.2in}
\noindent {\bf Acknowledgements.} N.C. Phuc is supported in part by Simons Foundation, award number 426071.


\vspace{.2in}
\begin{thebibliography}{xxxxxx}

\bibitem[Ad1]{Ad1} D. R. Adams, {\it Traces of potentials II}, Indiana Univ. Math. J. {\bf 22} (1972/73), 907--918. 

\bibitem[Ad2]{Ad2} D. R. Adams, {\it Sets and functions of finite $L^p$-capacity}, Indiana Univ. Math. J. {\bf 27} (1978), 611--627. 


\bibitem[Ad3]{Ad3} D. R. Adams,  {\it Quasi-additivity and sets of finite $L^p$-capacity}, Pacific J. Math. {\bf 79} (1978), no. 2, 283--291.

\bibitem[Ad4]{Ad4} D. R.  Adams, {\it Choquet integrals in potential theory}, Publ. Mat. {\bf 42} (1998), no. 1, 3--66.

\bibitem[AH]{AH} D. R. Adams  and   L. I. Hedberg, {\it Function Spaces and Potential
Theory}, Springer-Verlag, Berlin, 1996.
	
\bibitem[AM]{AM} D. R. Adams  and   N. G. Meyers, {\it Thinness and Wiener criteria for non-linear potentials}, Indiana Univ. Math. J. {\bf 22} (1972), 169--197.	
	
	
\bibitem[AX1]{AX1} D. R. Adams and J.  Xiao,  {\it Nonlinear analysis on Morrey spaces and their capacities}, Indiana Univ. Math. J. {\bf 53} (2004), 1629--1663.

\bibitem[AX2]{AX2} D. R. Adams and J.  Xiao, {\it Morrey spaces in harmonic analysis}, Ark. Mat. {\bf 50} (2012), no. 2, 201--230.

\bibitem[AP1]{AP1}  K. Adimurthi  and N. C. Phuc, {\it An end-point global gradient weighted estimate for quasilinear equations in non-smooth domains}, Manuscripta Math. {\bf 150} (2016),  111--135.
	
\bibitem[AP2]{AP2} K. Adimurthi and C. P.  Nguyen, {\it Quasilinear equations with natural growth in the gradients in spaces of Sobolev multipliers}, Calc. Var. Partial Differential Equations {\bf 57} (2018), no. 3, Art. 74, 23 pp. 
	
\bibitem[BRV]{BRV} O. Blasco, A. Ruiz, and L. Vega, {\it Non-interpolation in Morrey-Campanato and block spaces}, Ann. Scuola Norm. Sup. Pisa Cl. Sci. (4) {\bf 28} (1999), no. 1, 31--40.
	
\bibitem[Cho]{Cho} G. Choquet, {\it Theory of capacities}, Ann. Inst. Fourier (Grenoble)
{\bf 5} (1953--1954), 131--295.
			
\bibitem[CW]{CW} R. R. Coifman and G. Weiss, {\it Extensions of Hardy spaces and their use in analysis}, Bull. Amer. Math. Soc. {\bf 83} (1977), no. 4, 569--645. 
			
\bibitem[Den]{Den} D. Denneberg,  {\it Non-additive measure and integral}, Theory and Decision Library. Series B: Mathematical and Statistical Methods {\bf 27}. Kluwer Academic Publishers Group, Dordrecht (1994), x+178 pp.
			
\bibitem[Ger]{Ger} P. Germain, {\it Multipliers, paramultipliers, and weak-strong uniqueness for the Navier-Stokes equations},
J. Differential Equations, {\bf 226} (2006),  373--428.	
	
	
\bibitem[HMV]{HMV} K. Hansson,  V. G.  Maz'ya, and  I. E.  Verbitsky,  {\it Criteria of solvability for multidimensional Riccati equations}, Ark. Mat. {\bf 37} (1999), no. 1, 87--120.	

\bibitem[HKM]{HKM} J. Heinonen, T. Kilpel\"ainen,  and   O. Martio,  {\it Nonlinear Potential
	Theory of Degenerate Elliptic Equations}, Oxford Univ. Press, Oxford, 1993.

\bibitem[ISY]{ISY}	T. Izumi, E. Sato, and  K. Yabuta, {\it Remarks on a subspace of Morrey spaces},
Tokyo J. Math. {\bf 37} (2014), 185--197.



\bibitem[KV]{KV} N. J. Kalton and I. E. Verbitsky,  {\it Nonlinear equations and weighted norm inequalities}, Trans. Amer. Math. Soc. {\bf 351} (1999), no. 9, 3441--3497.
	
\bibitem[Kom]{Kom} J. Koml\'os, {\it A generalization of a problem of Steinhaus}, Acta Math. Acad. Sci. Hungar. {\bf 18} (1967) 217--229.

\bibitem[Lan]{Lan} N. S. Landkof,  {\it Foundations of modern potential theory}, Translated from the Russian by A. P. Doohovskoy. Die Grundlehren der mathematischen Wissenschaften, Band 180. Springer-Verlag, New York-Heidelberg (1972), x+424 pp.	
	
\bibitem[L-R]{L-R} P.G. Lemari\'e-Rieusset, {\it Recent Developments in the Navier-Stokes Problem},
Chapman \& Hall, CRC Press, Boca Raton (2002).


\bibitem[LT]{LT} J. Lindenstrauss and L Tzafriri, {\it Classical Banach spaces II. Function spaces}. Ergebnisse der Mathematik und ihrer Grenzgebiete [Results in Mathematics and Related Areas], 97. Springer-Verlag, Berlin-New York, 1979. x+243 pp. 

\bibitem[Lux]{Lux} W. A. J. Luxemburg, {\it Banach function spaces}, Thesis, Technische Hogeschool te Delft (1955), 70 pp.

\bibitem[MST]{MST} M. Mastylo, Y. Sawano,  and H. Tanaka,  {\it Morrey-type space and its K\"othe dual space}, Bull. Malays. Math. Sci. Soc. {\bf 41} (2018), no. 3, 1181--1198.	

	
\bibitem[MH]{MH}	V. G.  Maz'ja and V. P. Havin,  {\it A nonlinear potential theory}, Uspehi Mat. Nauk {\bf 27} (1972), no. 6, 67--138 (in Russian). English translation: Russ. Math. Surv. {\bf 27} (1972), 71--148.


	

 	
	
\bibitem[MS1]{MS1} V. G. Maz'ya and T.O. Shaposhnikova, {\it The Theory of Multipliers in Spaces of Differentiable Functions},
Pitman, New York (1985).

\bibitem[MS2]{MS2} V. G. Maz'ya and T.O. Shaposhnikova, {\it
Theory of Sobolev Multipliers. With Applications to Differential and Integral Operators}, Grundlehren der Mathematischen Wissenschaften, vol. {\bf 337}, Springer-Verlag, Berlin (2009), p. xiv+609.


\bibitem[MV]{MV} V. G. Mazʹya and I. E. Verbitsky, {\it Capacitary inequalities for fractional integrals, with applications to partial differential equations and Sobolev multipliers}, Ark. Mat. {\bf 33} (1995),  81--115.

\bibitem[MP]{MP} T. Mengesha and N. C.  Phuc, {\it Global estimates for quasilinear elliptic equations on Reifenberg flat domains}, Arch. Ration. Mech. Anal. {\bf 203} (2012),  189--216. 



\bibitem[Mey]{Mey} N. G. Meyers, {\it A theory of capacities for potentials of functions in Lebesgue classes}, Math. Scand. {\bf 26} (1970) 255--292.
	
\bibitem[NP]{NP} Q.-H. Nguyen and N. C. Phuc, {\it Good-$\lambda$ and Muckenhoupt-Wheeden type bounds in quasilinear measure datum problems, with applications},  Math. Ann.
{\bf 374} (2019), 67--98.	

\bibitem[OV]{OV} J. Orobitg and J.  Verdera, {\it Choquet integrals, Hausdorff content and the Hardy-
	Littlewood maximal operator}, Bull. Lond. Math. Soc. {\bf 30} (1998), 145--150.

\bibitem[Ph1]{Ph1} N. C. Phuc, {\it Weighted estimates for nonhomogeneous quasilinear equations with discontinuous coefficients} Ann. Sc. Norm. Super. Pisa Cl. Sci. (5) {\bf 10} (2011),  1--17.
	
\bibitem[Ph2]{Ph2} N. C. Phuc,  {\it Quasilinear Riccati type equations with super-critical exponents}, Comm. Partial Differential Equations {\bf 35} (2010), no. 11, 1958--1981.	
	
 \bibitem[Ph3]{Ph3} N. C. Phuc, {\it  Nonlinear Muckenhoupt-Wheeden type bounds on Reifenberg flat domains, with applications to quasilinear Riccati type equations}, Adv. Math. {\bf 250} (2014), 387--419. 
 	
\bibitem[PhPh]{PhPh} T. V.  Phan and N. C. Phuc, {\it Stationary Navier-Stokes equations with critically singular external forces: existence and stability results}, Adv. Math. {\bf 241} (2013), 137--161.

\bibitem[PhV]{PhV} N. C. Phuc, and I. E.  Verbitsky, {\it Quasilinear and Hessian equations of Lane-Emden type}, Ann. of Math. (2) {\bf 168} (2008), no. 3, 859--914.  	
 	
\bibitem[RT]{RT} M. Rosenthal and H. Triebel	{\it Calder\'on-Zygmund operators in Morrey spaces}, Rev. Mat. Complut.  {\bf 27} (2014), 1--11.
 	
 \bibitem[Ryc]{Ryc} V. S. Rychkov, {\it Littlewood-Paley theory and function spaces with $A_p^{\rm loc}$ weights}, Math. Nachr. {\bf 224} (2001), 145--180. 

 \bibitem[ST]{ST}	Y. Sawano and H. Tanaka,  {\it The Fatou property of block spaces}, J. Math. Sci. Univ. Tokyo {\bf 22} (2015), no. 3, 663--683.
 	
\bibitem[VW]{VW} I. E. Verbitsky and R. L. Wheeden,  {\it Weighted trace inequalities for fractional integrals and applications to semilinear equations}, J. Funct. Anal. {\bf 129} (1995), no. 1, 221--241. 	
	

\end{thebibliography}
\end{document}